\documentclass[12pt]{amsart}
\usepackage{amsmath}
\usepackage{amssymb}
\usepackage[all]{xy}
\usepackage{mathdots}
\usepackage{tikz}

\numberwithin{equation}{section}

\newtheorem{thm}{Theorem}[section]

\newtheorem{prop}[thm]{Proposition}
\newtheorem{cor}[thm]{Corollary}
\newtheorem{rem}[thm]{Remark}

\newtheorem{exam-nota}[thm]{Example-Notation}
\newtheorem{nota}[thm]{Notation}
\newtheorem{dfn}[thm]{Definition}

\newtheorem{dfn-nota}[thm]{Definition-Notation}

\newtheorem{dfn-lem}[thm]{Lemma-Definition}

\newcommand{\beqa}{\begin{eqnarray*}}
\newcommand{\eeqa}{\end{eqnarray*}}

\newcommand{\Pflag}{{\mathcal P}}

\newcommand{\Orb}{{\mathcal O}}

\newcommand{\id}{\mbox{${\rm id}$}}

\newcommand{\ft}{\mbox{${\mathfrak t}$}}
\newcommand{\fk}{\mbox{${\mathfrak k}$}}
\newcommand{\fg}{\mbox{${\mathfrak g}$}}

\newcommand{\fl}{\mbox{${\mathfrak l}$}}
\newcommand{\fs}{\mbox{${\mathfrak s}$}}
\newcommand{\fsl}{\mbox{${\fs\fl}$}}

\newcommand{\fh}{\mbox{${\mathfrak h}$}}

\newcommand{\fp}{\mbox{${\mathfrak p}$}}
\newcommand{\fr}{\mbox{${\mathfrak r}$}}
\newcommand{\fb}{\mbox{${\mathfrak b}$}}

\newcommand{\fm}{\mbox{${\mathfrak m}$}}
\newcommand{\fu}{\mbox{${\mathfrak u}$}}

\newcommand{\dw}{\dot{w}}
\newcommand{\dtw}{\dot{\tilde{w}}}

\newcommand{\eps}{\epsilon}

\newcommand{\PR}{\mbox{${\mathbb P}$}}

\newcommand{\C}{\mbox{${\mathbb C}$}}

\newcommand{\Ad}{{\rm Ad}}

\newcommand{\fgl}{\mathfrak{gl}}

\newcommand{\B}{\mathcal{B}}

\newcommand{\fso}{\mathfrak{so}}
\newcommand{\fsp}{\mathfrak{sp}}
\newcommand{\hv}{\hat{v}}

\newcommand{\ms}{m(s_{\alpha})}

\newcommand{\dotw}{\dot{w}}

\newcommand{\tildeQ}{{\widetilde{Q}}}
\newcommand{\tildeQK}{{\widetilde{Q}}_K}

\newcommand{\rkk}{\mbox{rank}(\fk)}

\title{$B_{n-1}$-bundles on the flag variety, I}

\author[M. Colarusso]{Mark Colarusso}
\address{Department of Math and Stats, University of South Alabama, Mobile, Al, 36608}
\email{mcolarusso@southalabama.edu}

\author[S. Evens]{Sam Evens}
\address{Department of Mathematics, University of Notre Dame, Notre Dame, IN, 46556}
\email{sevens@nd.edu}

\begin{document}
\maketitle

\begin{abstract}
In this paper, we consider the orbits of a Borel subgroup $B_{n-1}$ of $G_{n-1}=GL(n-1)$ (respectively $SO(n-1)$) acting on the flag variety $\mathcal{B}_{n}$
of $G=GL(n)$ (resp. $SO(n)$).  The group $B_{n-1}$ acts on $\B_{n}$ with finitely many orbits, and we use the known description of $G_{n-1}$-orbits on $\B$ to study these orbits.  In particular, 
we show that a $B_{n-1}$-orbit is a fibre bundle over a $B_{n-1}$-orbit in a generalized flag variety of $G_{n-1}$ with fibre a $B_{m-1}$-orbit on $\B_{m}$ for some $m<n$.  We further use this 
fibre bundle structure to study the orbits inductively and describe the monoid action on the fibre bundle.  In a sequel to this paper, we use these results to give a complete combinatorial model for these orbits
and show how to understand the closure relations on these orbits in terms of the monoid action. 
\end{abstract}

\section{Introduction}\label{s:intro}

In this paper, we consider certain solvable subgroups acting on the
flag variety with finitely many orbits.   The subgroups we consider arise naturally in the complex Gelfand-Zeitlin integrable system and in the theory
of Gelfand-Zeitlin modules.   The goal of this paper, and its sequel,
is to provide a combinatorial description of the orbits and in particular
to understand monoidal actions on the orbits.  We then hope to use these
results to understand representations of the Lie algebra realized via
the Beilinson-Bernstein correspondence using local systems on these orbits.

In more detail, let $G_n = GL(n)$ or $SO(n)$, and embed $G_{n-1}$
as a symmetric subgroup of $G_{n}$, up to center in the $GL(n)$-case.   We note that the pairs $(G_n,G_{n-1})$ are (up to center and isogeny) the multiplicity free symmetric pairs.
We let $B_{n-1}$ denote a Borel subgroup of $G_{n-1}$ and recall that 
$B_{n-1}$ acts on the flag variety of Borel subgroups $\B_n$ of $G_n$
with finitely many orbits.  Our main result is that $B_{n-1}$-orbits
on $\B_n$ fibre over $B_{n-1}$-orbits on a generalized flag variety
of $G_{n-1}$ with fibre given by $B_{m-1}$-orbits on $\B_m$ for some
$m < n$.   More precisely, we construct a class of parabolic subgroups
of $G_n$, which we call {\it special parabolic subgroups}.   

\begin{thm}\label{thm:mainintro}
For each $B_{n-1}$-orbit $Q$ on $\B_n$, there is a special parabolic
$R$ of $G_n$ so that $R \cap G_{n-1}$ is a parabolic subgroup of $G_{n-1}$,
and 
\begin{enumerate}
\item there is a fibre bundle $Q \to Q_1$, where $Q_1$ is a $B_{n-1}$-orbit
on $G_{n-1}/R \cap G_{n-1}$;
\item the fibre of $Q \to Q_1$ may be identified with a $B_{m-1}$-orbit $Q_2$
on $\B_m$.
\end{enumerate}
If we denote the above $B_{n-1}$-orbit by $Q = \mathcal{O}(Q_1, Q_2)$, then
$\mathcal{O}(Q_1, Q_2)=\mathcal{O}(Q_1^{\prime}, Q_2^{\prime})$ if and only if $Q_i = Q_i^{\prime}$
for $i=1, 2.$
\end{thm}

Our theorem allows for an inductive description of $B_{n-1}$-orbits
on the flag variety $\B_n$ in terms of special parabolics,
Weyl group data, and data derived inductively from the same problem
for a smaller index.  The special parabolic subalgebras are in one-to-one correspondence with $G_{n-1}$-orbits on $\B_{n}$.
For the closed $G_{n-1}$-orbits on $\B_{n}$, the corresponding special parabolics are the representatives of these orbits given in Part (2) of Propositions \ref{prop:typeAflag}, \ref{prop_sooddflag}, and \ref{prop_soevenflag}.  
The special parabolics corresponding to $G_{n-1}$-orbits that are not closed are given in the orthogonal cases by taking
parabolics containing the upper triangular matrices and corresponding
to the choice of simple roots $\{ \alpha_i,\dots, \alpha_{l}\}$ where $i=1,\dots, l$ for type B 
and $i=1,\dots, l-1$ for type D using the usual ordering of
 simple roots (see \eqref{eq:Levi2} and \eqref{eq:Levi1}).  In the general linear case the special parabolics  are
 given by choosing a consecutive subset of simple roots and a corresponding element of the Weyl group (see \eqref{eq:rijflag}).


In a sequel to this paper, we use our results to give a combinatorial description of the $B_{n-1}$-orbits on $\B_n$.   Further, we use the action of a monoid using both simple roots of $G_{n-1}$ and of $G_n$ to show that every $B_{n-1}$-orbit arises from a zero dimensional $B_{n-1}$-orbit using the monoid action.   Richardson and Springer \cite{RS} have pointed out the seminal role of verifying this result in the description of the orbit closure relation.  As a consequence of work of Richardson and Springer \cite{RS}, we use this result to determine the closure relation between $B_{n-1}$-orbits in terms of the monoid action.   An essential step in this process is to describe instances when this monoid action may be described efficiently using the fibre bundle description of the orbits of Theorem \ref{thm:mainintro} (see Theorems \ref{thm:Kintertwiners}, \ref{thm:rightbundleaction}, and \ref{thm:openinter}).

In the $GL(n)$-case, Hashimoto \cite{Hashi} describes the 
 $B_{n-1}$-orbits on $\B_n$, and discusses part of the monoid
action that we consider in this paper.   Our work advances on \cite{Hashi} in that our approach also describes orbits both in general linear and orthogonal cases in an essentially uniform manner,  elucidates the bundle structure, and gives a complete description of the closure relation. While Hashimoto's approach is based on analyzing orbits of a Borel subgroup, our approach is based more on studying the orbits of $G_{n-1}$ on the flag manifold.   It would also be interesting to relate our results to the work of Gandini and Pezzini \cite{GP}, who study the general case of orbits of solvable subgroups of $G$ with finitely many orbits on the flag variety of $G$.   In particular, it would be of interest to understand some of the invariants they consider in their more general context in our situation.

In Section \ref{ss:notation} of this paper, we introduce notation and recall standard facts about orbits and embeddings of $G_{n-1}$ in $G_n$ for later use.   In Section \ref{s:bundlestructure}, we establish the fibre bundle structure for $B_{n-1}$-orbits and use it to give an inductive method to describe $B_{n-1}$-orbits on $\B_n$.   In Section \ref{s:monoid}, we discuss the monoid action and compute it in cases where the role of the bundle structure is salient.   The remaining cases are treated in the sequel to this paper using a more explicit combinatorial description of the orbits.

We would like to thank Friedrich Knop, Jacopo Gandini, and Guido Pezzini for useful discussions which aided in the development of our results.   During the preparation of this paper, the first author was supported in part by the National Security Agency grant H98230-16-1-0002 and the second author was supported in part by the Simons Foundation grant 359424.

\section{Preliminaries}\label{ss:notation}
We discuss some notation and standard results that will be used throughout
the paper.

\subsection{Notation}\label{ss:papernotations}

By convention, all algebraic groups and Lie algebras discussed in this paper have complex coefficients.   If an algebraic group
$M$ is defined, we denote by the corresponding Gothic letter $\fm$ its
Lie algebra, and if a Lie subalgebra $\fm$ of the Lie algebra of a complex
Lie group $G$ is given, we let $M$ denote the connected subgroup of $G$
with Lie algebra $\fm$.  For an algebraic group $M$ and $g\in M$, we
denote by $\Ad(g)$ both the adjoint action of $M$ on its Lie algebra, and
the conjugation action of $M$ on itself.
As a convention, we let $\B_{\fm}$ denote the flag variety of Borel subalgebras
of the Lie algebra $\fm$, and let $\B_n = \B_{{\fg}_n}$ for ${\fg}_n =\fgl(n)$ or $\fso(n)$.   For $(V,\beta)$ a vector space with symmetric bilinear
form $\beta$ and subspace $U$, we let $U^{\perp} := \{ v \in V : \beta(u,v)=0 \ \forall u \in U \}.$   Throughout the paper, we let $\imath$ denote a fixed choice of $\sqrt{-1}$.  

\subsection{Realizations}\label{ss:realization}

We use the conventions for realizing $\fso(n)$ from Chapters 1 and 2 of \cite{GW}
(see also \cite{CE19}).   More explicitly, we let $\fso(n)$ be the symmetry Lie algebra 
of  the non-degenerate, symmetric bilinear form $\beta$ on $\C^{n}$
given by $\beta(x,y)= x^{T} J y$, where $J$ is the automorphism of
$\C^n$ such that $J(e_i)=e_{n+1-i}$ for the standard basis $e_1, \dots, e_n$
of $\C^n$ and $i=1, \dots, n$. Note that $\beta(e_j, e_{n+1-j})=1$ for
$j=1, \dots, n$.  We let $SO(n)$ be the connected algebraic
group with Lie algebra $\fso(n).$  Then the subalgebra $\fh_n$ consisting
of diagonal matrices in $\fso(n)$ is a Cartan subalgebra, and the subalgebra
$\fb_+$ consisting of upper triangular matrices in $\fso(n)$ is a Borel
subalgebra, and we refer to these as the standard Cartan and Borel subalgebras.
We refer also to $\fh_n$ and $\fb_+$ to denote the
diagonal and upper triangular matrices in $\fgl(n).$   
For $\fg=\fgl(n)$ or $\fso(n)$, any parabolic subalgebra $\fp\supset\fb_{+}$ we 
refer to as a standard parabolic subalgebra.  We also let $\Phi=\Phi(\fg,\fh_{n})$ denote 
the roots of $\fg$ with respect to the Cartan subalgbra $\fh_{n}$.  
If $\alpha\in \Phi$, we let $\fg_{\alpha}$ denote the root 
space corresponding to the root $\alpha$. 

It will be convenient to relabel the standard basis 
$\{e_{1},\dots, e_{n}\}$ of $\C^{n}$ depending on whether $n$ is even or odd.  
If $n=2l$ is even, then we let $e_{-i}:=e_{2l+1-i}$ for $i=1,\dots, l$.  Note 
that for $j,\, k\in \{-l,\dots,-1,1,\dots, l\}$, we have $\beta(e_{j}, e_{k})=\delta_{j,-k}$.  
If $n=2l+1$, then we let $e_{0}:=e_{l+1}$ and $e_{-i}:=e_{2l+2-i}$ for $i=1,\dots, l$.  
For $j,\, k\in\{-l,\dots, -1,0,1,\dots, l\}$, we have $\beta(e_{j}, e_{k})=\delta_{j,-k}$.  

We recall the real rank one subalgebras for $\fgl(n)$ and $\fso(n)$.
For $\fg=\fgl(n)$, let $t=\mbox{diag}[1,\dots, 1, -1] \in GL(n)$, and let $\theta(x)=txt^{-1}$ for $x\in\fg$.  Then the fixed set 
$\fg^{\theta}=\fgl(n-1)\oplus\fgl(1)$, regarded as block matrices, and $\fg_{n-1} = \fgl(n-1) \oplus 0$.

For $\fg=\fso(2l+1)$, let $t$ be an element of the Cartan subgroup with
Lie algebra $\fh_{2l+1}$  with the property that 
$\Ad(t)|_{\fg_{\alpha_i}}=\id$ for $i=1, \dots, l-1$ and $\Ad(t)|_{\fg_{\alpha_l}}=-\id$.  Consider the involution $\theta_{2l+1}:=\Ad(t)$.   Then 
$\fk:=\fg^{\theta_{2l+1}}\cong \fso(2l)$, realized as block matrices with the
$l+1$st rows and columns equal to $0$ 
(see \cite{Knapp02}, p. 700).  Note that $\fh_{2l+1}=\fh_{2l} \subset \fk$.
For later use, let $\sigma_{2l+1}$ be the diagonal matrix with respect to the
basis $\{ e_{1},\dots, e_{l}, e_{0}, e_{-l},\dots, e_{-1} \}$ such that $\sigma_{2l+1}\cdot e_j = e_j$
for $j\not= 0$ and $\sigma_{2l+1}\cdot e_{0} = -e_{0}$ and note that
$\theta = \theta_{2l+1} =  \Ad(\sigma_{2l+1})$.

In the case $\fg=\fso(2l)$, consider the involutive 
linear map $\sigma_{2l}:\C^{2l} \to \C^{2l}$
such that $\sigma_{2l}(e_l)= e_{-l}$ and $\sigma_{2l}(e_i)=e_i$ for $i\not= l$ or $-l.$
Define an involution $\theta$ of $\fg$ by $\theta = \Ad(\sigma_{2l})$.
Then $\fk:=\fg^{\theta_{2l}}\cong\fso(2l-1)$,
and $\theta_{2l}$ is the involution induced by the diagram automorphism interchanging
the simple roots $\alpha_{l-1}$ and $\alpha_l$ (see \cite{Knapp02}, p. 703).  If 
we realize $\fh_{2l}$ as diagonal matrices of the form
$x=\mbox{diag}[x_1, \dots, x_l, -x_l, \dots, -x_1]$ and define for $i=1, \dots,
l$, $\eps_i \in \fh_{2l}^*$ by $\eps_i(x)=x_i$, then the induced action of
$\theta_{2l}$ is via
 $\theta_{2l}(\eps_l)=-\eps_l$ and $\theta_{2l}(\eps_i)=\eps_i$ for $i=1, \dots, l-1$.  
We will omit the subscripts $2l+1$ and $2l$ from $\theta$
when $\fg$ is understood.

We also denote the involution of the corresponding classical group $G$ by $\theta$.  For 
$G=GL(n)$, $G^{\theta}\cong GL(n-1) \times GL(1)$ is connected, while for 
$G=SO(n)$, the group $G^{\theta}\cong S(O(n-1)\times O(1))$ is disconnected.  
We let $K:=(G^{\theta})^{0}$ be the identity component of $G^{\theta}$.  Then $K\cong SO(n-1)$, and $K$ has Lie algebra $\fk=\fg^{\theta}$. 


\begin{rem}\label{rem:sominus2}
For $SO(n)$, note that we can realize the subgroups $SO(n-1)$ and $SO(n-2)$ as follows.
For $n$ either odd or even, note that the $-1$ eigenspace $(\C^n)^{-\sigma_n}$
is one dimensional and is nondegenerate for the restriction of $\beta$.
If  $(\C^n)^{-\sigma_n} = \C v_n$, then since the perpindicular of an $SO(n-1)$-stable subspace is $SO(n-1)$-stable, it follows that
\begin{equation} \label{e:sominus1}
SO(n-1) = \{ g \in SO(n) : g\cdot v_n= v_n \}.
\end{equation}
  In the case $n=2l,$
 let $V_2 = \C e_l + \C e_{-l}.$   Then 
\begin{equation} \label{e:soevenless2}
SO(2l-2) = \{ g \in SO(2l) : g\cdot v = v \ \forall v \in V_2 \}.
\end{equation}
    Here $SO(2l-2)$ is
realized as determinant one orthogonal transformations of the orthogonal
complement $V_2^{\perp}$ to $V_2$.  
In the case $n=2l+1$, one similarly observes that
\begin{equation} \label{e:sooddless2}
SO(2l-1) = \{ g \in SO(2l+1) : g\cdot e_{0}=e_{0}, \ g\cdot (e_l - e_{-l})
= e_l - e_{-l}.\}
\end{equation}
For later use, we also note that if $\alpha_1, \dots, \alpha_l$ are the
simple roots for $SO(2l+1)$, then the simple roots for $SO(2l)$ are given
by $\alpha_1, \dots, \alpha_{l-1}, \alpha_{l-1} + 2\alpha_l$.  We may take
as root vectors for the last two simple roots of the Lie algebra $\fso(2l)$ the vectors
 $e_{\alpha_{l-1}}=E_{l-1,l}-E_{-l,-(l-1)}$ and $e_{\alpha_{l}}=E_{l-1,-l}-E_{l,-(l-1)}$ and then their sum
$e_{\alpha_{l-1}}+e_{\alpha_{l}}$ is the root vector for
for $\fso(2l-1)$ for the simple root $\alpha_{l-1}$.
\end{rem}


\subsection{Description of $K$-orbits on $\B_n$ in the multiplicity free
setting}\label{ss:Korbits}

Note that the cases $(\fg, \fg_{n-1})=(\fgl(n), \fgl(n-1))$ or $(\fso(n), \fso(n-1))$
are up to center essentially the only situations where each irreducible representation
of $\fg$ has multiplicity free branching law when regarded as a representation
of a symmetric subalgebra $\fk$.  By abuse of notation, we use $\fk$ to denote $\fg_{n-1}$ and $K$ to denote $G_{n-1}$.   This property is well-known to be equivalent to the property
that a Borel subgroup $B_K$ of $K$ has finitely many orbits on $\B_{\fg}$, the
variety of Borel subalgebras of $\fg$ (see for example, Proposition 3.1.3
of \cite{CE19}).  
 
 \begin{nota}\label{nota:orbitnotation}
 For a group $A$ acting on a variety $X$, we let $A\backslash X$ denote the collection of  $A$-orbits on $X$
 \end{nota}

We will need to use an explicit description of $K\backslash \B_{\fg}$.
We present this mostly without proof, and note that it follows from explicit
study of the Richardson-Springer discussion of orbits in \cite{RS}, and
is worked out explicitly for $GL(n)$ in Section 4.4 of \cite{CEexp} (see especially
Example 4.16, and Examples 4.30 to 4.32)
and for $SO(n)$ in section 2.7 of \cite{CEspherical}.

We first recall the monoidal action.  Suppose $M$ is a subgroup of $G$
acting with finitely many orbits on the flag variety $\B_{\fg}$ of Borel
subalgebras of $\fg$, the Lie algebra of connected reductive group $G$.    
For a simple root $\alpha$ for $G$, consider the ${\PR}^1$-bundle 
$p:\B_{\fg} \to {\Pflag}_{\alpha, \fg}$, where ${\Pflag}_{\alpha, \fg}$ consists of the parabolic subalgebras of $\fg$ with Lie algebra $G$-conjugate
 to $\fb_0 + \fg_{-\alpha}$,
 where $\fb_0$ is a Borel subalgebra for which $\alpha$ is a simple root.
 We call the parabolic subalgebras in ${\Pflag}_{\alpha, \fg}$ simple parabolics of type $\alpha$.
Then if $Q=M\cdot \fb$ is a $M$-orbit in $\B_{\fg}$, $p^{-1}p(Q)$ contains
a unique open $M$-orbit, and we call this orbit $m(s_{\alpha})*Q$
and note that if $m(s_{\alpha})*Q\not= Q$, then $\dim(m(s_{\alpha})*Q)=\dim(Q)+1.$   
We let $\fp_{\alpha}=p(\fb)$ and let $P_{\alpha}$ denote the corresponding
parabolic subgroup of $G$ with unipotent radical $U_{\alpha}$ and if $T$ is a maximal torus of the group $B$ with Lie algebra $\fb$,
let $L_{\alpha}$ be the Levi subgroup of $P_{\alpha}$ containing $T$ and let $Z_{\alpha}$ be the center of $L_{\alpha}$.
We denote by $M_\alpha$ the image of
$M\cap P_{\alpha}$ in the group $S_{\alpha}:= P_{\alpha}/Z_{\alpha}U_{\alpha}$ locally isomorphic to $SL(2)$.
 Then one of the following occurs.

\begin{enumerate}
\item Complex Case: $M_{\alpha}$ contains a maximal unipotent subgroup of $S_{\alpha}$ and is contained
in a Borel subgroup of $S_{\alpha}$.  In this case, $m(s_{\alpha})*Q$ meets
$p^{-1}p(\fb)$ in a $M_{\alpha}$-orbit isomorphic to $\C$ and the complement is a point;
\item Noncompact/Real Case: $M_{\alpha}$ is locally isomorphic to a maximal torus of $S_{\alpha}$.  In this case, $m(s_{\alpha})*Q$ meets
$p^{-1}p(\fb)$ in a $M_{\alpha}$-orbit isomorphic to $\C^{\times}$, and if $M_{\alpha}$ is a torus, the complement is a union of two $M_{\alpha}$-orbits, and otherwise, the complement is a single $M_{\alpha}$-orbit;
\item Compact Case: $M_{\alpha} = S_{\alpha}.   $   In this case, $m(s_{\alpha})*Q=Q$ and
$p^{-1}p(\fb) \cong \PR^1$ is a single $M_{\alpha}$-orbit.
\end{enumerate}

\begin{dfn}\label{dfn:roottype}
\begin{enumerate}
\item In the complex case above, we say that the simple root $\alpha$ is 
{\it complex stable} for $Q$ if $m(s_{\alpha})*Q \not= Q$ and 
{\it complex unstable} for $Q$ if $m(s_{\alpha})*Q = Q;$
\item In the noncompact/real case above, we say that the simple root $\alpha$
is {\it noncompact } for $Q$ if $m(s_{\alpha})*Q \not= Q$ and {\it real} for $Q$
if $m(s_{\alpha})*Q=Q;$
\item In the compact case above, we say the simple root $\alpha$ is
{\it compact} for $Q$, and note that in this case $m(s_{\alpha})*Q=Q.$
\end{enumerate}
\end{dfn}

The type of a simple root $\alpha$ for $Q$ depends only on $Q$ and not on
the choice of $\fb$.  We introduce some notation needed to describe
$K$-orbits on $\B_{\fg}.$

\begin{rem} \label{dfn:roottypetheta}
  When $M$ is the fixed point set of an involution $G$, then our description of roots by type coincides with the usual definition from real groups \cite{RS}.
  \end{rem}

\begin{nota}\label{nota:cayley}
For each simple root $\alpha$, choose root vectors $e_\alpha \in \fg_\alpha$,
$f_\alpha \in \fg_{-\alpha}$, and $h_\alpha=[e_\alpha, f_\alpha]$ so that
 $e_{\alpha}, f_{\alpha}, h_{\alpha}$ spans a subalgebra 
of $\fg$ isomorphic to $\fsl(2)$. 
Consider the unique Lie algebra homomorphism $\phi_\alpha:\fsl(2) \to \fg$ such that:
\begin{equation}\label{eq:sl2map}
\phi_{\alpha}:\left[\begin{array}{cc} 0 & 1\\
0 & 0\end{array}\right]\to e_{\alpha},\;    
\phi_{\alpha}:\left[\begin{array}{cc} 0 & 0\\
1 & 0\end{array}\right]\to f_{\alpha},\;
\phi_{\alpha}:\left[\begin{array}{cc} 1 & 0\\
0 & -1\end{array}\right]\to h_{\alpha}\; 
\end{equation}
Let $\phi_\alpha:SL(2) \to G$  be the induced Lie group homomorphism,
and let 
    \begin{equation}\label{eq:cayley}
u_{\alpha}=\phi_{\alpha}\left(\frac{1}{\sqrt{2}}\left[\begin{array}{cc} 1 & \imath\\
\imath & 1\end{array}\right]\right).
\end{equation}
\end{nota}

In the remainder of this section, we consider the cases
where $K=G_{n-1}$ as above, and consider the monoid
action relative to the $K$-action on $\B_n$. We now introduce some notation which will be used
throughout the paper.

 \begin{nota} \label{nota:standard}
  Let 
   $$
  \mathcal{F}=( V_{0}=\{0\}\subset V_{1}\subset\dots\subset V_{i}\subset\dots\subset V_{n}=\C^{n})
   $$
be a full flag in $\C^{n}$, with $\dim V_{i}=i$ and $V_{i}=\mbox{span}\{v_{1},\dots, v_{i}\}$, with each $v_{j}\in\C^{n}$.  We will denote this flag $\mathcal{F}$ 
by
   $$
  \mathcal{F}=  (v_{1}\subset  v_{2}\subset\dots\subset v_{i}\subset v_{i+1}\subset\dots\subset v_{n}). 
   $$

We will also make use of the following notation for partial flags in $\C^{n}$.  Let
   $$
   \mathcal{P}=(V_{0}=\{0\}\subset V_{1}\subset\dots\subset V_{i}\subset\dots\subset V_{k}=\C^{n})
   $$
   denote a $k$-step partial flag with $\dim V_{j}=i_{j}$ and $V_{j}=\mbox{span}\{v_{1},\dots, v_{i_{j}}\}$ for $j=1,\dots, k$.  
   Then we denote $\mathcal{P}$ as 
   $$
  \mathcal{P}=( \{v_{1},\dots, v_{i_{1}}\} \subset \{v_{i_{1}+1},\dots, v_{i_{2}}\}\subset\dots\subset \{v_{i_{k-1}+1},\dots, v_{i_{k}}\}).
   $$
\end{nota}
 
\begin{nota}\label{nota:Weyl}

 Let $T$ be the maximal torus of $G$  with Lie algebra $\ft$, and let $W$ be the Weyl group with respect \text{t}o $T$. 
 For an element $w\in W$, let $\dot{w}\in N_{G}(T)$ be a representative of $w$.   
 If $\fm\subset\fg$ is a Lie subalgebra which is normalized by $T$, then $\Ad(\dot{w})\fm$ is independent of the choice of representative of $w$,
 and we denote it by $w(\fm)$. Similarly, we denote by $w(M)$ the corresponding group.
\end{nota}

\begin{prop}\label{prop:typeAflag}
Let $\fg=\fgl(n)$ and $\fk=\fgl(n-1)$.  
\begin{enumerate}
\item The number of $K$-orbits on $\B$ is ${n+1\choose 2}$.  
\item For $i=1,\dots, n$, let $w_{i}$  be the cycle $(n, n-1, \dots i)$ in the symmetric group $\mathcal{S}_n$, 
and let $\fb_{i}:=w_{i}(\fb_{+})$.  The distinct closed $K$-orbits on $\B$ are the $K$-orbits $Q_{i}=K\cdot \fb_{i}$ so there are exactly $n$ closed $K$-orbits.  Further, the Borel subalgebra $\fb_{i}$ is 
the stabilizer of the flag:
\begin{equation}\label{eq:flagi}
\mathcal{F}_{i}:=(e_{1}\subset\dots \subset e_{i-1}\subset\underbrace{e_{n}}_{i}\subset e_{i}\subset \dots \subset e_{n-1}). 
\end{equation}
\item The non-closed $K$-orbits are of the form:
\begin{equation}\label{eq:Qijorbs}
Q_{i,j}=m(s_{\alpha_{j-1}})\dots m(s_{\alpha_{i}})* Q_{i}
\end{equation}
with $1\leq i< j\leq n$.  The codimension of $Q_{i,j}$ is $n-1-(j-i)$.  
Further, the Borel subalgebra:
\begin{equation}\label{eq:bij}
\fb_{i,j}:=\Ad(\hat{v})(\fb_+) \in Q_{i,j}, \ \ \hat{v} :={\dot{w}}_{i} \cdot u_{\alpha_{i}} \cdot  \dot{s}_{\alpha_{i+1}}\dots \dot{s}_{\alpha_{j-1}}, 
\end{equation} 
and $\fb_{i,j}$ is the stabilizer of the flag:
\begin{equation}\label{eq:flagij}
\mathcal{F}_{i,j}:=(e_{1}\subset\dots\subset e_{i-1} \subset \underbrace{e_{i}+e_{n}}_{i}\subset e_{i+1}\subset\dots\subset e_{j-1}\subset\underbrace{e_{n}}_{j}\subset e_{j}\subset\dots\subset e_{n-1}). 
\end{equation}
In particular, the unique open orbit is $Q_{1,n}$, and it contains the Borel subalgebra 
$$
\fb_{1,n}=\Ad(w_{1})\Ad(u_{\alpha_{1}})s_{\alpha_{2}}\dots s_{\alpha_{n-1}}(\fb_{+})
$$
which stabilizes the flag:
\begin{equation}\label{eq:openflag}
\mathcal{F}_{1,n}:=(e_{1}+e_{n}\subset e_{2}\subset\dots\subset e_{n-1}\subset e_{n}). 
\end{equation}

\end{enumerate}
\end{prop}

\par\noindent The following diagram indicates the $K$-orbits on $\B_4$,
together with the order relations given by closure.  For general $n$, the diagram has
the same shape, with $n$ closed orbits, $n-1$ orbits one line below, and so forth, until we have one orbit on the last line.

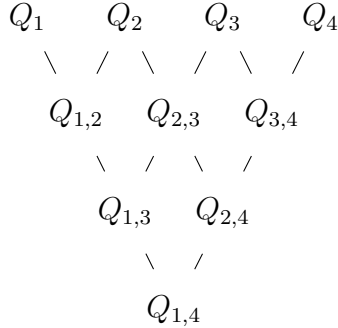
\hspace{11em}
\begin{tikzpicture}  
  [scale=1.3,auto=center,every node/.style={circle,fill=white!20}] 
\node (a1) at (5,4) {$Q_1$};
\node (a2) at (6,4) {$Q_2$};
\node (a3) at (7,4) {$Q_3$};
\node (a4) at (8,4) {$Q_4$};
\node (a5) at (5.5,3) {$Q_{1,2}$};
\node (a6) at (6.5,3) {$Q_{2,3}$};
\node (a7) at (7.5,3) {$Q_{3,4}$};
\node (a8) at (6,2) {$Q_{1,3}$};
\node (a9) at (7,2) {$Q_{2,4}$};
\node (a10) at (6.5,1) {$Q_{1,4}$};

\draw (a1) -- (a5); 
  \draw (a2) -- (a5);  
  \draw (a2) -- (a6);  
  \draw (a3) -- (a6);  
  \draw (a3) -- (a7);  
  \draw (a4) -- (a7);  
  \draw (a5) -- (a8);  
\draw (a6) -- (a8);  
\draw (a6) -- (a9); 
\draw (a7) -- (a9);  
\draw (a8) -- (a10);  
\draw (a9) -- (a10);  
 
\end{tikzpicture}

We now classify the orbits of $K=SO(n-1)$ on the flag variety for $\fg = \fso(n)$.  When $n=2l+1$ is odd, there is a bijection between $\B_{\fso(n)}$ and the set of all 
maximal isotropic flags in $\C^n$, i.e., partial flags of the form
$$
\mathcal{I}\mathcal{F}=( V_{1}\subset\dots\subset V_{i}\subset\dots\subset V_{l}),
$$
with $\dim(V_i)=i$ and $B(u,w)=0$ for all $u,w \in V_i.$   When $n=2l$ is
even, the space of maximal isotropic flags in $\C^n$ consists of two $SO(2n)$-orbits.   To distinguish them, we choose a maximal isotropic flag
$U_1 \subset \dots \subset U_l$, where $U_i$ is the span of the standard
basis $e_1, \dots, e_i$ from Section \ref{ss:realization}.   Then we identify
$\B_{\fso(n)}$ with the maximal isotropic flags $V_1 \subset \dots \subset V_l$
such that $\dim(V_l \cap U_l) \equiv l \pmod{2}.$ 
  As before, if
$V_i$ is the span of vectors $\{ v_1, \dots, v_i \},$ we write the above maximal isotropic  flag
as
$$
\mathcal{I}\mathcal{F}=(v_1 \subset v_2 \subset \dots \subset v_{[\frac{n}{2}]}).
$$

We consider the type $B$ and $D$ cases separately, and use the enumeration
of simple roots in \cite{GW}.

\begin{prop} \label{prop_sooddflag}
Let $\fg=\fso(2l+1)$ and $\fk=\fso(2l)$.
\begin{enumerate}  
\item There are exactly  $l+2$ $K$-orbits on the flag variety $\B$ of $\fg$.
\item We let $\fb_-:=s_{\alpha_l}(\fb_+)$.    Exactly two $K$-orbits on $\B$ are
closed, and they are $Q_+:= K\cdot \fb_+$ and $Q_-:=K\cdot \fb_-.$
Further, $m(s_{\alpha_l})*Q_+ = m(s_{\alpha_l})* Q_-=K\cdot \Ad(u_{\alpha_l})\fb_+.$
\item The non-closed orbits are of the form $$Q_i:=m(s_{\alpha_{i+1}})* m(s_{\alpha_{i+2}})
\cdots m(s_{\alpha_{l-1}}) *m(s_{\alpha_l})* Q_+$$ for $i=0, \dots, l-1$.
Moreover, the codimension of $Q_i$ in $\B$ is $i$.    Further, 
\begin{equation}\label{eq:sooddboreli}
\fb_i = \Ad(\hat{v})(\fb_+) \in Q_i, \ \hat{v} := u_{\alpha_{l}} \dot{s}_{\alpha_{l-1}}  \dot{s}_{\alpha_{l-2}} \dots \dot{s}_{\alpha_{i+1}}.
\end{equation}
In particular, the unique open $K$-orbit contains the Borel 
subalgebra 
\begin{equation}\label{eq:openborelodd}
\fb_{0}=\Ad(u_{\alpha_{l}})s_{\alpha_{l-1}}  s_{\alpha_{l-2}} \dots s_{\alpha_{1}}  (\fb_+).
\end{equation}
\end{enumerate}
\end{prop}

\begin{prop}\label{prop_soevenflag} 
Let $\fg=\fso(2l)$ and $\fk=\fso(2l-1)$.
\begin{enumerate}
\item There are exactly $l$ $K$-orbits in the flag variety $\B$ of $\fg$.
\item  The orbit $Q_+:= K\cdot \fb_+$ is the only closed $K$-orbit.
\item Let $$Q_i:=m(s_{\alpha_i}) \dots m(s_{\alpha_{l-1}})*Q_+$$
and let 
\begin{equation}\label{eq:soevenboreli}
\fb_i = \Ad(\hat{v})(\fb_+), \ \hat{v} := {\dot{s}}_{\alpha_{l-1}} {\dot{s}}_{\alpha_{l-2}} \dots  {\dot{s}}_{\alpha_{i}}(\fb_+)\mbox{ for } i=1,\dots, l-1.
\end{equation}
Then $Q_i=K\cdot \fb_i$ has codimension $i-1$ in $\B$.   The distinct $K$-orbits are 
$Q_+, Q_{l-1}, \dots, Q_1$.
 In particular, the unique open orbit is $Q_{1}$ and contains the Borel subalgebra 
\begin{equation}\label{eq:openboreleven}
\fb_{1}=s_{\alpha_{l-1}} s_{\alpha_{l-2}} \dots  s_{\alpha_{1}}(\fb_+).
\end{equation}
\end{enumerate}
\end{prop}

\begin{rem}\label{rem:soflagnotation}
In case $\fg=\fso(2l)$ or $\fso(2l+1),$ the Borel subalgebra $\fb_+$
corresponds to the maximal isotropic flag $(e_1 \subset e_2 \subset \dots \subset e_l)$.
By Proposition \ref{prop_soevenflag}, the Borel subalgebra
$\fb_1$ in the open $K$-orbit in  $\B_{\fso(2l)}$ corresponds to the maximal isotropic flag
$(e_l \subset e_1 \subset \dots \subset e_{l-1}).$
Similarly, by Proposition \ref{prop_sooddflag}, it follows that the Borel subalgebra
$\fb_0$ in the open $K$-orbit in $\B_{\fso(2l+1)}$ corresponds to the maximal isotropic flag
$(u_{\alpha_l}(e_l) \subset e_1 \subset e_2 \subset \dots \subset e_{l-1}).$
We may choose our root vectors so that  
$\C \cdot u_{\alpha_l}(e_l)=\C \cdot (e_l + \imath\sqrt{2} e_{0} + e_{-l}).$
We consider the $\alpha_l$ root vector $e=E_{l,0} - E_{0,-l}$ and take $f$ to be
the transpose of $e$, and let $x=[e,f]$, and note that $\{ e, x, f \}$
generates the three dimensional subalgebra ${\fs}_{\alpha_l} := \phi_{\alpha_l}(\fsl(2)).$
Let $V_3$ be the span of $e_l, e_{0},$ and $e_{-l}$, and note that
$V_3$ is a representation of ${\fs}_{\alpha_l}$ and is isomorphic to the
adjoint representation via an isomorphism mapping $e_{-l} \to f$,
$e_{0} \to -x$, and $e_l \to e.$  A brief computation using the composed isomorphism $\fsl(2) \to \fso(3) \to V_3$ above then
gives the above formula for $u_{\alpha_l}(e_l)$.
\end{rem}





\section{Structure of $B_{n-1}$-orbits} \label{s:bundlestructure}
 
Recall that $G=GL(n)$ or $SO(n)$ and $K=G_{n-1}$  from Section \ref{ss:realization}.   We show that
$B_{n-1}$-orbits on the flag variety $\B$ of $G$ have a useful bundle structure
and describe the data required to classify orbits.

\subsection{Fibre Bundle Structure} \label{ss:fibrebundle}

Let $Q_{K}=K\cdot \fb \subset \B$ be a $K$-orbit.  Let $R$ be a parabolic
subgroup of $G$ containing $B$ and consider the projection
$\pi:\B \to \mathcal{R} := G/R$.   Then the morphism
\begin{equation}\label{eq:easybundle} 
\phi:K\times_{K\cap R} R/B \to \pi^{-1}\pi(Q_K),\  \ \phi(k,\fb_1) = \Ad(k)\cdot \fb_1,
\end{equation}
is an isomorphism, and the preimage $\phi^{-1}(Q_K)$ is $K\times_{K\cap R} (K\cap R)\cdot \fb.$  This realizes both $\pi^{-1}\pi(Q_K)$ and $Q_K$ as bundles over
$K/K\cap R.$    Let $R=L\cdot U$ be a Levi decomposition of the
parabolic $R$.   Recall that the morphism $R\cdot \fb \to \B_{\fl}$
given by $\fb_1 \mapsto \fb_1 \cap \fl$ is an isomorphism, where $\fb_1 \in R\cdot \fb \cong \B_R.$

We now assume that $R$ is $\theta$-stable.   Then by Theorem 2 of \cite{BH},
there is a $\theta$-stable Levi subgroup $L$ of $R$, $K\cap R$ is a parabolic
subgroup of $K$, and $K\cap R = (K\cap L)\cdot (K\cap U)$ is a Levi decomposition of $K\cap R$.  The group $K\cap R$ then acts on $\B_{\fl}$ through its quotient
$K\cap L$, and it follows that we can identify $Q_K$ with 
$K\times_{K\cap R} (K\cap L)\cdot (\fb \cap \fl).$  Since $K\cap L$ is a 
symmetric subgroup of $L$, $K\cap L$ acts with finitely many orbits
on $\B_{\fl}.$

We now prove our basic result about $K$-orbits on $\B$ for the multiplicity
free  pairs $(G,K)$.

 \begin{thm}\label{thm:specials}
 Let $Q_{K}=K\cdot \fb$ be a $K$-orbit on $\B$.  Then there exists a $\theta$-stable parabolic subgroup  $R \supset B$ with $\theta$-stable 
 Levi decomposition $R=L\cdot U$ such that the pair $(\fl,\fl\cap\fk)$ is of the same type as $(\fg,\fk)$ up to abelian factors in the center of 
$\fl$ and the $K\cap L$-orbit 
 $(K\cap L)\cdot(\fb\cap\fl)$ is open in the flag variety $\B_{\fl}$.  Thus, 
  \begin{equation}\label{eq:QKbundle}
 Q_{K}\cong K\times_{K\cap R} (K\cap L)\cdot (\fb\cap \fl).
 \end{equation}  
 is a $K$-homogenous fibre bundle over the partial flag variety $K/K\cap R $ with  fibre the open $K\cap L$-orbit on $\B_{\fl}$.
 \end{thm}
 
\begin{proof}
If $Q_{K}$ is closed, then $B$ is $\theta$-stable by Corollary 6.6 of \cite{Sp},
and then $\fl$ is abelian so that $\B_{\fl}= \{ \fb \cap \fl \}$ is a point,
and the assertion follows.

When $Q_K$ is not closed, we consider the different types $A, B,$ and $D$
separately.  In type A, by Proposition \ref{prop:typeAflag},
the orbit $Q_{K}=Q_{i,j}=K\cdot\fb_{i,j}$ for some $i, j$ with $1 \le i < j \le n$, and $\fb_{i,j}$ is the Borel subalgebra in Equation (\ref{eq:bij}).  Let $P_{i,j}$ be the standard 
parabolic subgroup of $GL(n)$ which is the stabilizer of the flag
\[
{\Pflag}_{i,j} := (e_1 \subset \dots \subset e_{i-1} \subset \{e_i, \dots, e_j \}\subset e_{j+1} \subset\dots\subset e_n)
\]
and corresponds to the choice of simple roots $\{ \alpha_i, \dots, \alpha_{j-1} \}.$    Recall the cycle  $w_i = (n, n-1, \dots i)$ from 
Proposition \ref{prop:typeAflag}(2) and let $\fr =\fr_{i,j}:={w}_{i}(\fp_{i,j})$.  
Then $\fr_{i,j}$ is the Lie algebra stabilizer of 
 the partial flag $w_{i}({\Pflag}_{i,j})$ which is 
\begin{equation}\label{eq:rijflag}
\mathcal{R}_{i,j}:=(e_{1}\subset \dots\subset e_{i-1}\subset\{e_{i},e_{i+1},\dots, e_{j-1}, e_{n}\}\subset e_{j}\subset e_{j+1}\subset \dots\subset e_{n-1}.)
\end{equation}
By Proposition \ref{prop:typeAflag}(3), 
 $\fb_{i,j}$ is the stabilizer of the full
flag in Equation (\ref{eq:flagij}).  It follows that $\fb_{i,j}$ also stabilizes the partial flag $\mathcal{R}_{i,j}$ in (\ref{eq:rijflag}), so that $\fb_{i,j}$ is in the parabolic $\fr_{i,j}.$  Since the Cartan subalgebra $\fh_n$ is contained in $\fr_{i,j}$, it follows that $\fr_{i,j}$ is $\theta$-stable.  Further, 
$\fr_{i,j}$ has Levi decomposition $\fr_{i,j}=\fl\oplus\fu$, where $\fl\cong \fgl(1)^{i-1}\oplus \fgl(j-i+1)\oplus \fgl(1)^{n-j}.$  
It follows from (\ref{eq:rijflag}) that $\fr_{i,j}\cap\fk$ is a standard parabolic subalgebra of $\fk$ with Levi factor 
$\fl\cap\fk \cong \fgl(1)^{i-1}\oplus \fgl(j-i)\oplus \fgl(1)^{n-j}.$

To prove the Theorem when $\fg=\fgl(n)$, it remains to show that $(K\cap L)\cdot (\fb_{i,j}\cap\fl)$ is open in the flag variety of $\fl$.  In 
the Levi decomposition of $\fr_{i,j}$ above, $\fgl(j-i+1)$ is identified with 
endomorphisms of the span of $\{ e_{i},\dots, e_{j-1}, e_{n}\}.$  Observe that 
$\fb_{i,j}\cap \fgl(j-i+1)$ is the stabilizer of the flag 
$$
(e_{i}+e_{n}\subset e_{i+1}\subset \dots\subset e_{j-1}\subset e_{n})
$$
in $\fgl(j-i+1).$ By 
part (3) of Proposition \ref{prop:typeAflag}, the point $\fb_{i,j}\cap\fgl(j-i+1)$ is in the open
$K\cap GL(j-i+1)$-orbit on $\B_{\fgl(j-i+1)}$.  
Hence, $\fb_{i,j}\cap\fl$ is in the open $K\cap L$-orbit on 
the flag variety $\B_{\fl}$ of $\fl$.

For the type B case,  the relevant symmetric pair is 
$(\fso(2l+1),\fso(2l))$.  Recall from Proposition \ref{prop_sooddflag}
that since $Q_K$ is not closed, $Q_K=Q_i$ with $i<l$ where
$Q_i=K\cdot \fb_{i}=\Ad(u_{\alpha_{l}})s_{\alpha_{l-1}}\dots s_{\alpha_{i+1}}(\fb_{+})$.  
Let $\fr_i\subset \fg$ be the standard parabolic subalgebra generated by $\fb_{+}$ and the
negative simple root spaces $\fg_{-\alpha_{l}}, \fg_{-\alpha_{l-1}},\dots, \fg_{-\alpha_{i+1}}.$  Thus, 
\begin{equation}\label{eq:Levi2}
\fr_i=\fl\oplus\fu\mbox{ with } \fl\cong\fso(2(l-i)+1) \oplus \fgl(1)^{i}.
\end{equation}
Note that $\fr_i$ is $\theta$-stable and that $\theta|_{[\fl,\fl]}=\theta_{2(l-i)+1}$.  
Thus, $\fr_i\cap\fk$ is a parabolic subalgebra of $\fk$ with Levi 
factor:
\begin{equation}\label{eq:anotherLevi}
\fl\cap\fk\cong \fso(2(l-i)) \oplus \fgl(1)^{i},
\end{equation}
so that $K\cap L\cong SO(2(l-1))\times GL(1)^{i}$.   To see that $\fb_i\subset \fr_i$, note that we can choose the representative $\dot{s}_{\alpha_{j}}$ of $s_{\alpha_{j}}$ so that $\dot{s}_{\alpha_{j}}\in L$ for $j=i+1,\dots, l$, and 
$u_{\alpha_{l}}\in L$ by Equation (\ref{eq:cayley}).  Thus, the element 
\begin{equation}\label{eq:v}
v:=u_{\alpha_{l}}\dot{s}_{\alpha_{l-1}}\dots \dot{s}_{\alpha_{i+1}}\in L\subset R.  
\end{equation}
Hence, $\fb_{i}=\Ad(v)\fb_{+}\subset \Ad(v)\fr_i=\fr_i$.

It remains to show that $(K\cap L)\cdot (\fb\cap\fl)$ can be identified with the open 
$K\cap L$-orbit in the flag variety $\B_{\fl}$.  
Note that $\fb_{+}\cap \fl$ can be identified with the standard Borel 
subalgebra $\fb_{+,\fl}$ of upper triangular matrices 
in $\fl$.  Since the element $v$ in Equation (\ref{eq:v}) is in $L$, we have:
$$
\fb\cap\fl=(\Ad(v)\fb_{+})\cap\fl=\Ad(v)(\fb_{+}\cap\fl)=\Ad(v)\fb_{+,\fl}.  
$$
It follows from Equations (\ref{eq:openborelodd}) and (\ref{eq:v})
 that $\Ad(v)\fb_{+,\fl}\subset \B_{\fl}$ is a representative 
of the open $K\cap L$-orbit on $\B_{\fl}$.

The type D case is similar to type B. 
In this setting, $\fg=\fso(2l)$ and $\fk=\fso(2l-1)$, and since $Q_K$ is
a $K$-orbit that is not closed, then by Proposition
\ref{prop_soevenflag}, $Q_{K}=Q_i$ has codimension $i-1<l-1$.  
By part (3) of Proposition \ref{prop_soevenflag}, we can 
take $\fb=\fb_{i}=s_{\alpha_{l-1}}s_{\alpha_{l-2}}\dots s_{\alpha_{i}}(\fb_{+})$. 
Let $\fr_i\subset \fg$ be the standard parabolic subalgebra generated by $\fb_{+}$ and the
negative simple root spaces $\fg_{-\alpha_{l}}, \fg_{-\alpha_{l-1}},\dots, \fg_{-\alpha_{i}}.$  
Then $\fr_i$ has the Levi decomposition:
\begin{equation}\label{eq:Levi1}
\fr_i=\fl\oplus\fu\mbox{ with } \fl\cong\fso(2(l-i)+2) \oplus \fgl(1)^{i-1}.
\end{equation}
We claim that $\fr_i$ is $\theta$-stable.  Indeed, we saw in Section \ref{ss:realization} that $\theta(\eps_{l})=-\eps_{l}$ 
and $\theta(\eps_{k})=\eps_{k}$ for $k\neq l$.  It follows that $\theta(\alpha_{i})=\alpha_{i}$  for $i=1,\dots, l-2$  and that $\alpha_{l-1}$ and $\alpha_{l}$
are complex $\theta$-stable with $\theta(\alpha_{l-1})=\alpha_{l}$.  It then follows easily that $\fr_i$ has Levi decomposition (\ref{eq:Levi1}) and that 
$\theta|_{\fl_{ss}}=\theta_{2(l-i)+2}$, whence $\fl^{\theta}=\fk\cap\fl\cong \fso(2(l-i)+1)\oplus \fgl(1)^{i-1}$, and
$(L^{\theta})^{0}=K\cap L\cong SO(2(l-i)+1)\times GL(1)^{i-1}$.  The remainder of the proof proceeds as in the previous case.

\end{proof}

\begin{rem}\label{r:bhanalogue}
In Proposition 12 of the paper \cite{BH}, Brion and Helminck prove a closely related result which in the multiplicity free case, describes $K$-orbits on $\B$
as bundles over a closed $K$-orbit with fibre an open orbit in a smaller flag
variety.  In the orthogonal cases, our result coincides with the result of \cite{BH},
but in the general linear case, our bundle is in many cases different from theirs.
\end{rem}

\begin{rem}\label{nota:special}
We call a parabolic of $\fg$ a \emph{special parabolic subalgebra} if it is one of the representatives of the closed $K$-orbits given in Part (2) of Propositions 
\ref{prop:typeAflag}, \ref{prop_sooddflag}, and \ref{prop_soevenflag} or is one of the parabolic subalgebras $\fr_{i,j}\supset\fb_{i,j}$ and $\fr_{i}\supset \fb_{i}$ constructed in 
the proof of Theorem \ref{thm:specials}.  We also call the corresponding parabolic subgroups the {\it special parabolic subgroups} of $G$.  Note that
for $\fr$ a special parabolic subalgebra of $\fg$, the standard Borel subalgebra $\fb_{n-1}$ of $\fk$ is in $\fr$.   Hence, $\fr\cap\fk$ is a standard parabolic subalgebra of
$\fk$.
\end{rem}

\begin{rem}\label{r:typeC}
Theorem \ref{thm:specials}  fails for other real rank one symmetric pairs.  
For example, if $\fg=\fsp(6)$, and $\fk=\fsp(4)\oplus\fsp(2)$, then there exists a $K$-orbit $Q_{K}=K\cdot \fb$ where 
the minimal $\theta$-stable parabolic subalgebra containing $\fb$ is all of $\fg$.  However, $Q_{K}$ is not open.

\end{rem}

For $\theta$-stable parabolics $R$, $K\cap R$ is a parabolic of $K$ and
we let $W_{K}^{K\cap R}\subset W_K$ be  the representatives of shortest length for the cosets in $W_{K}/W_{K\cap R}$. 
Now we use the fibre bundle structure of a $K$-orbit $Q_{K}$ given in (\ref{eq:QKbundle}) to 
describe the $B_{n-1}$-orbits $Q$ contained in $Q_{K}$ as fibre bundles.

 \begin{thm}\label{thm:Bkbundles}
 Let $Q_{K}=K\cdot\fb$ be a $K$-orbit in $\B$.  
 Let $\fr\supset\fb$ be a special parabolic subalgebra, and let $R\subset G$ be the 
 corresponding parabolic subgroup with $\theta$-stable Levi decomposition
$R=LU$.  
 Suppose $Q$ is a $B_{n-1}$-orbit contained in $Q_{K}$.  
 Then $Q$ fibres over a $B_{n-1}$-orbit in the partial flag variety $K/K\cap R$ of $R$ with 
 fibre isomorphic to a $B_{n-1}\cap L$-orbit contained in the open $K\cap L$ -orbit of the flag
 variety $\B_{\fl}$ of $\fl$.  More precisely,  
 \begin{equation}\label{eq:Bkbundle}
 Q\cong B_{n-1}\times_{B_{n-1}\cap w(R) } (B_{n-1}\cap w(L))\cdot \Ad(\dotw)\fb,
 \end{equation}
 where $w \in W_{K}^{K\cap R}$,
 \begin{equation}\label{eq:isBorel}
 B_{n-1}\cap w(L)=w(B_{n-1}\cap L)
 \end{equation}
 is a Borel subgroup of $K\cap w(L)$, and $\Ad(\dotw)\fb$ is contained 
 in the open $K\cap w(L)$-orbit in the flag variety $\B_{w(R)}\cong \B_{w(L)}$ of $R$.  
 
 Further, we have a one-to-one correspondence:
 \begin{equation}\label{eq:correspondence}
 B_{n-1}\backslash \mathcal{B}\leftrightarrow \displaystyle\coprod_{R \mbox{ special}} \left(B_{n-1}\backslash K/K\cap R \times B_{n-1}\cap L \backslash \tilde{Q}_{K\cap L}\right ),
\end{equation}
where $\tilde{Q}_{K\cap L}$ is the open $K\cap L$-orbit on $\B_{\fk\cap\fl}$ (see Notation \ref{nota:orbitnotation}).  
 \end{thm}

\begin{proof}
Consider the fibre bundle $Q_{K}=K\times_{K\cap R} (K\cap L)\cdot\fb$ in Equation (\ref{eq:QKbundle}).
Let $\pi|_{Q_{K}}: Q_{K}\to K/K\cap R =Q_{K,\fr}$ be the canonical projection.  
Let $Q = B_{n-1}\cdot \fb^{\prime}$ be a $B_{n-1}$-orbit in $Q_K$.  Its image under $\pi$ is 
a $B_{n-1}$-orbit in the partial flag variety $K/K\cap R$.  Since $ W_{K}^{K\cap R}$ indexes $B_{n-1}$-orbits in $K/K\cap R$,
 then 
$\pi(Q)=B_{n-1}\cdot w(\fr)$ for some $w\in W_{K}^{K\cap R},$
and we may assume that $\fb^{\prime}=w(\fb)$ by replacing $\fb$ by a $K\cap L$-conjugate.  Then the fibre bundle structure
from Equation (\ref{eq:easybundle}) identifies $Q$ as
\begin{equation}\label{eq:intermediatebundle}
Q\cong B_{n-1}\times_{B_{n-1}\cap  w(R)} (B_{n-1}\cap w(R)) \cdot \Ad(\dotw)(\fb).
\end{equation}
Now the group $B_{n-1}\cap w(R)$ acts on 
$ w(R)/ w(B)\cong \B_{w(\fl)}$ through its image 
in the quotient $w(R)/w(U)$.  Denote this image by $\hat{B}_{w(L)}$.  We claim that
\begin{equation}\label{eq:weclaim}
\hat{B}_{w(L)}=B_{n-1}\cap w(L)=w(L\cap B_{n-1}),
\end{equation}
so that $\hat{B}_{w(L)}$ is a Borel subgroup in $K\cap w(L)$.  To 
verify the claim, first consider $w(\fl)\cap\fb_{n-1}$.  
By Remark \ref{nota:special}, the special parabolic $\fr$ contains 
 $\fb_{n-1}$, so $\fb_{n-1}$  is a Borel subalgebra of $\fr\cap\fk$.
Let $\Phi^{+}_{\fl\cap\fk}$ be the roots appearing in the Borel
subalgebra $\fb_{n-1} \cap \fl$.  Since $w\in W_{K}^{K\cap R}$, standard
results imply that $w(\Phi^{+}_{\fl\cap\fk})\subset \Phi^{+}_{\fb_{n-1}}$.
Thus $w(\fl\cap\fb_{n-1})\subset w(\fl) \cap \fb_{n-1}$.  
It follows that $w(\fl)\cap\fb_{n-1}=w(\fl\cap\fb_{n-1})$, so that 
$w(\fl\cap\fb_{n-1})$ is a Borel subalgebra of $\fl\cap\fk$ with corresponding Borel subgroup 
$w(L\cap B_{n-1})=w(L)\cap B_{n-1},$ yielding the second equality in Equation (\ref{eq:weclaim}).  
To get the first equality of Equation (\ref{eq:weclaim}), we consider the Lie algebra $\hat{\fb}_{w(L)}=\mbox{Lie}(\hat{B}_{w(L)})$.  
Since $\fh_n\cap\fk\subset w(\fr)\cap\fk$, we have the decomposition
$$
w(\fr)\cap \fb_{n-1}=w(\fl)\cap\fb_{n-1}\oplus w(\fu)\cap \fb_{n-1}.
$$
Thus, $\hat{\fb}_{w(L)}=w(\fl)\cap\fb_{n-1}$ and Equation (\ref{eq:weclaim}) follows.  
Equation (\ref{eq:Bkbundle}) now follows from (\ref{eq:intermediatebundle}).  
Note that the orbit $(B_{n-1}\cap w(L))\cdot \Ad(\dotw)\fb$ in the fibre of the bundle in (\ref{eq:Bkbundle}) is isomorphic
via translation by $\dotw^{-1}$ to an orbit of the standard Borel subgroup $B_{n-1}\cap L$ of $L\cap K$ contained in the open 
$K\cap L$-orbit of $\B_{\fl}$.  

The one-to-one correspondence in Equation (\ref{eq:correspondence}) follows easily.

\end{proof}

\begin{nota}\label{nota:bundle}
We use the notation of Theorem \ref{thm:Bkbundles}.
Given a $B_{n-1}-$orbit $Q_{\fr}$ in $K/K\cap R$ and a $B_{n-1} \cap L$-orbit
$Q_{\fl}$ in the open $K\cap L$-orbit on $\B_{\fl}$, we let $\Orb(Q_{\fr}, Q_{\fl})$ be the corresponding
$B_{n-1}$-orbit on $\B$ from this theorem.  
Theorem \ref{thm:Bkbundles} implies that for any $K$-orbit $Q_{K}$ the 
$B_{n-1}$-orbits contained in $Q_{K}$ are of the form $\Orb(Q_{\fr}, Q_{\fl})$ as above. 
\end{nota}
  
 \begin{rem}\label{r:goodreps}
Let $\fb_{\hv}$ be one of the representatives for $\fb_{i,j}, \fb_i, \fb_{+}$, or $\fb_{-}$ for $K$-orbits
on $\B$ discussed in Propositions \ref{prop:typeAflag}, \ref{prop_sooddflag}, and \ref{prop_soevenflag}.   Let $Q$ be a 
$B_{n-1}$-orbit in one of the above 
$K$-orbits $Q_K$ and let $\fb_{\hv}=\Ad(\hv)\fb_+.$ 
 Then 
 \begin{equation}\label{eq:goodrep}
 Q=B_{n-1}\cdot \Ad(\dotw)\Ad(\ell) \fb_{\hv}=B_{n-1}\cdot \Ad(\dotw\ell\hv)\fb_{+}
 \end{equation}
 for a representative $\dotw$ of unique $w\in W^{K\cap R}_{K}$ and for $\ell \in K\cap L$ such that 
 $\Ad(\ell) (\fb_{\hv}\cap \fl)$ is a representative of the $B_{n-1}\cap L$-orbit $Q_{\fl}$ contained 
 in the open $K\cap L$-orbit on the flag variety $\B_{\fl}$ of $\fl$.  In the notation given above in \ref{nota:bundle}, 
 the orbit $Q_{\fr}=B_{n-1}\cdot w(\fr)$, and $Q_{\fl}=(B_{n-1}\cap L)\cdot\Ad(\ell)(\fb_{\hv}\cap\fl).$ 
 
 \end{rem}


By Equation (\ref{eq:correspondence}), the problem of describing the $B_{n-1}$-orbits on $G/B$ reduces to the description of $B_{n-1}$-orbits on partial flag
varieties $K/(K\cap R)$ via $W_{R\cap K}^K$ and the problem of determining
$B_{n-1}$-orbits in the open $K$-orbit in $\B$.  We now address this second
issue. 
 
 \subsection{$B_{n-1}$-orbits in the open $K$-orbit}\label{ss:bnopen}
 
Recall the pair $(G,K)$ with $G=G_n$ and $K=G_{n-1}$ from Section
\ref{ss:realization}.

\begin{thm}\label{thm:openorbit} Let ${\tilde{Q}}_K$ be the open $K$-orbit
on $\B_n$.    There are bijections
\begin{equation}\label{eq:correspondence2}
\{B_{n-1}-\mbox{orbits on } {\widetilde{Q}}_K\}\longleftrightarrow B_{n-1}\backslash K/ B_{n-2} \longleftrightarrow\{B_{n-2}-\mbox{orbits on } \B_{n-1}\},
\end{equation}
where the last equivalence is induced by the self-map $g\mapsto g^{-1}$ of $K$.
 \end{thm}

\begin{proof}
To prove the first equivalence, it suffices to find a point $\fb_y$ in ${\tilde{Q}}_K$ with stabilizer $B_{n-2}$ in $K$.   The second equivalence is then clear.
We prove the first equivalence by considering the three cases separately.
 
When $\fg=\fgl(n)$, by Proposition \ref{prop:typeAflag}, the open $K$-orbit
$Q_{1,n}$ on $\B_{n}$ is
the orbit through the flag 
$(e_1 + e_n \subset e_2 \subset \dots \subset e_{n-1} \subset e_n)$.  
Thus, $Q_{1,n}=GL(n-1)\cdot \mathcal{F}_y$, where 
$\mathcal{F}_y:=(e_{n-1} + e_n \subset e_1 \subset e_2 \subset \dots \subset e_{n-2} \subset e_n).$  Suppose that $g\in GL(n-1)$ fixes $\mathcal{F}_y.$  
For $j=1, \dots, n-1$, let $U_j$ be the span of $e_1, \dots, e_j$ and
note that $GL(n-1):= \{ g\in GL(n) : g\cdot e_n = e_n, \text{ and } \ g\cdot U_{n-1}=U_{n-1}\}.$   Since $g$ fixes the line $\C (e_{n-1} + e_n)$ in the flag
and $g\cdot e_n = e_n,$ we see that $g\cdot e_{n-1}=e_{n-1}.$
Since $g\in GL(n-1)$, then $g\cdot e_j \in U_{n-1}$  for $j=1, \dots, n-2$.
But also $g\cdot e_j \in \C (e_{n-1}+e_n) + U_j$ for $j=1, \dots, n-2$.
Thus, for $j=1, \dots, n-2$, the vector 
$g\cdot e_j \in U_{n-1} \cap ( \C (e_{n-1}+e_n) + U_j) = U_j.$
Hence, $g$ stabilizes $U_{n-2}$, so $g\in GL(n-2)$, and also
  $g$ stabilizes the standard flag 
$(e_1 \subset e_2 \subset \dots \subset e_{n-2})$ in $U_{n-2}$. 
 Thus, $g \in B_{n-2}$.
Since $B_{n-2}$ fixes $\mathcal{F}_y$, this proves the Theorem in
case of $GL(n)$.

Next, consider the case $\fg = \fso(2l)$ and recall the basis $e_1, \dots, e_{l}, e_{-l},\dots, e_{-1}$.   By Remark  \ref{rem:soflagnotation},
the open $SO(2l-1)$-orbit on $\B_{\fso(2l)}$ is identified with the
$SO(2l-1)$-orbit through the maximal isotropic flag
$\mathcal{I}\mathcal{F}_y := (e_l \subset  e_1 \subset \dots \subset e_{l-1}).$
Let $g\in SO(2l-1)_y$, the stabilizer of $\mathcal{I}\mathcal{F}_y.$
  Since $g\in SO(2l-1)$, then $g$ commutes with the element
 $ \sigma_{2l}$ on
$\C^{2l}$ (see Section \ref{ss:realization}).   Since $g$ stabilizes the flag
 $\mathcal{I}\mathcal{F}_y$, we see that
$g\cdot e_l = a e_l$ for some $a\in \C^{\times}$.   Since $g$ commutes with
$\sigma_{2l}$, it follows that $g\cdot e_{-l}= a e_{-l}$, and since
$\beta(e_l, e_{-l})=1$, the scalar $a=\pm 1$.  But $g \in SO(2l-1)$,
so by Equation \eqref{e:sominus1}, $g$ fixes the vector $e_l - e_{-l}$.
Thus, $a=1$ and by Equation \eqref{e:soevenless2}, $g\in SO(2l-2),$
so that $SO(2l-1)_y = SO(2l-2)_y.$
But $B_{2l-2}$ is clearly in $SO(2l-2)_y$ so since $SO(2l-2)_y$ is solvable and $B_{2l-2}$ is a Borel subgroup of $SO(2l-2),$ it follows
that $B_{2l-2}=SO(2l-2)_y=SO(2l-1)_y$, proving the required assertion.

Now let $\fg=\fso(2l+1)$ and recall the basis $e_1, \dots, e_{l}, e_{0}, e_{-l},\dots, e_{-1}$ from Section
\ref{ss:realization}.     By Remark \ref{rem:soflagnotation},
the open orbit is the $SO(2l)$-orbit through the point in $\B_{\fso(2l+1)}$
corresponding to the maximal isotropic flag 
$(e_l + \sqrt{2}\imath e_{0} + e_{-l} \subset e_1 \subset e_2 \subset \dots \subset e_{l-1}).$  We can choose an element $s$ of $SO(2l)$ in the diagonal torus that 
$s\cdot (e_l + \sqrt{2}\imath e_{0} + e_{-l})= \imath e_l + \sqrt{2} \imath e_{0}-\imath e_{-l}$,
so that the open orbit is the $SO(2l)$-orbit through the point $y$ in 
 $\B_{\fso(2l+1)}$ corresponding to the  maximal isotropic
flag $\mathcal{I}\mathcal{F}_y = (\imath e_l + \sqrt{2}\imath e_{0} - \imath e_{-l} \subset
e_1 \subset e_2 \subset \dots \subset e_{l-1})$.  Note that by a small 
calculation using the root vectors in $\fb_{2l-1}$, we see that
 $B_{2l-1}$ is contained in the stabilizer $SO(2l)_y$.  Now let 
$g \in SO(2l)_y$.  Since $g\in SO(2l)$, then $g\cdot e_{0}=e_{0}$.  Since $g$
stabilizes the line in the flag, we see
that $g\cdot (e_l - e_{-l})=e_l - e_{-l}$, so that $g\in SO(2l-1)$ by
Equation \eqref{e:sooddless2}.    As in the preceding paragraph, it follows
that $B_{2l-1}=SO(2l-1)_y$.
\end{proof}

 \begin{rem}\label{r:theycorrespond}
For later use, we make explicit the correspondence between $B_{n-1}$-orbits on the open $K$-orbit $\tildeQK$ in $\B_n$ and $B_{n-2}$-orbits 
on $\B_{n-1}$ given in Equation (\ref{eq:correspondence2}).  Let $Q=B_{n-1}\cdot \Ad(k)\tilde{\fb} \subset \tildeQK=K\cdot\tilde{\fb}$, where $\tilde{\fb}$ is the 
stabilizer of the (isotropic) flag $\mathcal{F}_{y}$ ($\mathcal{I}\mathcal{F}_{y}$) given in the proof of Theorem \ref{thm:openorbit}.  Denote 
the corresponding $B_{n-2}$-orbit in $\B_{n-1}$ by $Q^{op}$.  Note that $Q^{op}=B_{n-2}\cdot\Ad(k^{-1})\cdot \fb_{n-1}$.  We can realize this 
correspondence geometrically as follows.  Let $q:K\to \tildeQK\cong K/B_{n-2}$ and $\pi: K\to \B_{n-1}$ be the natural projections, 
and let $\psi:K\to K$ be the inversion map, i.e., $\psi(k)=k^{-1}$.  Then 
\begin{equation}\label{eq:correspondingorbit}
Q^{op}=\pi(\psi(q^{-1}(Q))).
\end{equation}
More generally, if $\mathcal{Y}\subset \tildeQK$ is any $B_{n-1}$-stable subvariety, it corresponds to a $B_{n-2}$-stable subvariety of $\B_{n-1}$:
\begin{equation}\label{eq:correspondingsubvariety}
\mathcal{Y}^{op}=\pi(\psi(q^{-1}(\mathcal{Y}))).
\end{equation}
Using (\ref{eq:correspondingsubvariety}), we note that 
\begin{equation}\label{eq:dimcorrespond}
\dim\mathcal{Y}^{op}=\dim \mathcal{Y}+\dim\fb_{n-2}-\dim \fb_{n-1}.
\end{equation}
 Note the easily verified numerical identity
 \begin{equation}\label{eq:dimclaim}
 \dim\fb_{n-1}-\dim\fb_{n-2}=\rkk. 
 \end{equation}
 Equation (\ref{eq:dimcorrespond}) then becomes
 \begin{equation}\label{eq:finaldim}
 \dim \mathcal{Y}^{op}=\dim \mathcal{Y}-\rkk. 
 \end{equation}
 \end{rem}
 \section{The monoid action for $B_{n-1}$-orbits on $\B_{n}$}\label{s:monoid} 
In this section, we use Theorems \ref{thm:Bkbundles} and Theorem \ref{thm:openorbit} to describe several monoid actions for the $B_{n-1}$-orbits on $\B_{n} = \B_{\fg_n}$.  One of the major points is to compute the monoid action on an
$B_{n-1}$-orbit $\Orb(Q_{\fr}, Q_{\fl})$ in $\B_n$ in terms of other monoidal
actions.  In later work, we apply this action to the study the order relation
on $B_{n-1}$-orbits on $\B_n$.

\subsection{Left and Right Monoid Actions }\label{ss:monoidactions}

Recall Notation \ref{nota:orbitnotation} and before let $K=G_{n-1}$.  Let $K_{\Delta}\subset G\times K$ be the diagonal copy of $K$ in the product 
$G\times K$.  Note that there is a canonical bijection
$$
\chi:B_{n-1}\backslash \B \to K_{\Delta}\backslash \B \times \B_K,$$
given by $\chi(Q)=K_{\Delta}\cdot (Q, eB_{n-1})$ and we let $Q_{\Delta}$ denote $\chi(Q)$.   Further,
$Q_{\Delta} \cong K\times_{B_{n-1}} Q$, and it follows that the map $Q \mapsto Q_{\Delta}$
preserves topological properties like closure relations and open sets, and $\dim(Q_{\Delta})=\dim(Q) + \dim(\B_{n-1}).$   

In particular, $K_{\Delta}$ has finitely many orbits on $\B \times \B_{n-1}$,
and hence the set $K_{\Delta}\backslash\B \times \B_{n-1}$ has 2 different monoid actions,
coming from the monoid actions for the two different kinds of simple roots
of $\fg \times \fk$ via the generalities from Section \ref{ss:Korbits}.  
We decompose the simple roots $\Pi_{\fg \times \fk} = \Pi_{\fk} \cup \Pi_{\fg}.$
\begin{nota}\label{r.notleftright}
For a simple root $\alpha \in \Pi_{\fk}$, we call the monoid
action $Q_{\Delta} \to m(s_{\alpha})*Q_{\Delta}$ a {\it left monoid action}.   If $\alpha
\in \Pi_{\fg}$, we call the monoid action $Q_{\Delta} \to m(s_{\alpha})*Q_{\Delta}$ a {\it right monoid action}.  We use the equivalence between
$B_{n-1}\backslash \B$ and $K_{\Delta}\backslash \B \times \B_{n-1}$ to similarly
define left and right monoid actions $m(s_{\alpha})$ on $B_{n-1}\backslash \B$.
\end{nota}

In more detail, for $\alpha \in \Pi_{\fk}$, the left action uses the
projection $p_{\alpha}:G/B_n \times K/B_{n-1} \to G/B_n \times K/P_{\alpha}$
where $P_{\alpha}$ is the parabolic of $K$ with Lie algebra $\fb_{n-1} + \fk_{-\alpha}.$   Thus, if $Q_{\Delta}=K_{\Delta}\cdot (Q, eB_{n-1})$, then
\[
p_{\alpha}^{-1}p_{\alpha}(Q_{\Delta}) = K_{\Delta} \cdot (Q, P_{\alpha}/B_{n-1})  = K_{\Delta}\cdot (P_{\alpha}\cdot Q ,e B_{n-1}).\]
It follows that for $\alpha \in \Pi_{\fk},$ and $Q=B_{n-1}xB_{n}/B_{n}$.
\begin{equation} \label{eq:leftmonoidactionreal}
m(s_{\alpha})*Q \text{ \ is the unique open } B_{n-1}-\text{orbit in \ }
P_{\alpha} \cdot Q=P_{\alpha} x B_{n}/B_{n}.
\end{equation}

Similar analysis shows that if $Q=B_{n-1}xB_n/B_{n}$ and $\alpha \in \Pi_{\fg}$, then 

\begin{equation}\label{eq:rightmonoid}
 m(s_{\alpha})*Q \mbox{ is the unique open }B_{n-1}-\mbox{orbit in } B_{n-1}xP_{\alpha}/B_n, 
\end{equation}
where $P_{\alpha}$ here is
the parabolic of $G$ whose Lie algebra is $\fb_n + \fg_{-\alpha}.$


For a useful generalization of the right monoid action, we recall the notion of strongly orthogonal roots.
\begin{dfn}\label{dfn:strongortho}
Two nonproportional roots $\alpha,\, \beta\in \Phi(\fg,\fh)$ are said to be \emph{strongly orthogonal} if $\alpha\pm\beta$ is not a root.
\end{dfn}
Recall that by elementary arguments, strongly orthogonal roots are necessarily orthogonal and the corresponding $\fsl(2)$ subalgebras, $\fs_{\alpha}$ and $\fs_{\beta}$, commute.
We also make use of the following notation.
\begin{nota}\label{nota:standards}
Let $S \subset \Pi_{\fg}$ be a subset of the simple roots of $\fg$, and let
$\fm_S \subset \fg$ be the standard Levi subalgebra of $\fg$ given by the 
subset $S$.  We denote the corresponding standard 
parabolic subalgebra by $\fp_{S}\supset\fb_{+}$ and its Levi decomposition by $\fp_{S}=\fm_{S}\oplus\fu_{S}$, and if $S$ is understood, we let $\fm = \fm_{S}$
and $\fu = \fu_{S}.$  Let $P_S, M_S, U_S$ be the corresponding connected 
subgroups of $G.$
\end{nota}
\begin{rem}\label{r:stronglyorthogonal}
Suppose $S\subset \Pi_{\fg}$ is a subset of simple roots, 
let $\alpha\in\Pi_{\fg}$ be a simple root which is strongly orthogonal to the roots in $S$, 
and let $S^{\prime}=S\cup\{\alpha\}$.  The projection $G/P_{S}\mapsto G/P_{S^{\prime}}$ 
is a $\mathbb{P}^{1}$-bundle.  Hence, if $H$ is a subgroup of $G$ acting on
$G/P_S$ with finitely many orbits, there is a right monoid action
$m(s_\alpha)$ on $H\backslash G/P_S.$
\end{rem}

\begin{rem}\label{r:bnonemonoid}
There are additional monoid actions we will use.  Let $\alpha \in \Pi_{\fk}$,
and consider the finite set $B_{n-1}\backslash K/P)$ where $P$ is a standard parabolic subgroup of $K$.   As above, there is a bijection
\[ 
\chi:B_{n-1}\backslash K/P \to K_{\Delta}\backslash K/B_{n-1} \times K/P
\]
given by $\chi(Q)=K_{\Delta}\cdot (eB_{n-1}, Q)=: Q^{n-1}_{\Delta}.$
This bijection satisfies similar properties as the above bijection also
denoted $\chi$.   \par\noindent (1) It follows that if $\alpha \in \Pi_{\fk}$ is a simple
root of $\Pi_{\fk \times \fk}$ with root space in the first factor, then there is another left monoid action $m(s_{\alpha})$ on $K\backslash K/B_{n-1} \times K/P$
and hence by the equivalence of orbits induced by this version of $\chi$,
there is a left monoid action $m(s_{\alpha})$ on $B_{n-1}\backslash K/P$.
Similar analysis to the previous left monoid action verifies the following
assertion.   For $Q \in B_{n-1}\backslash K/P$,

\begin{equation}\label{e:smallleftmonoid}
m(s_{\alpha})*Q \text{ \ is the unique open orbit in \ } P_{\alpha}\cdot Q.
\end{equation}

\par\noindent (2) There is also a right monoid action in the strongly orthogonal setting.  Again, we let $\alpha\in\Pi_{\fk}$.  
The standard parabolic $P\subset K$ is  $P=P_{S}$ for $S\subset \Pi_{\fk}$.  
Suppose $\alpha$ is strongly orthogonal to the roots of $S$.  Then if the root space of $\alpha\in \Pi_{\fk\times\fk}$ is in the second factor, 
it follows from Remark \ref{r:stronglyorthogonal} that there is a right monoid action $m(s_{\alpha})$ 
on $K_{\Delta} \backslash K/B_{n-1} \times K/P$ and hence on $B_{n-1}\backslash K/P$ via the bijection $\chi$.   
It can be shown that if $Q=B_{n-1}\cdot x P_{S}/P_{S}$, then 
\begin{equation}\label{eq:smallrightmonoid}
m(s_{\alpha})*Q \text{ \ is the unique open orbit in \ } B_{n-1} x P_{S^{\prime}}/P_{S},
\end{equation}
where $P_{S^{\prime}}$ is the standard parabolic subgroup of $K$ corresponding to the subset of 
simple roots $S^{\prime}=S \cup\{\alpha\}$.
\end{rem}

\begin{rem}\label{r.abusemonoidnotation}
We will abuse notation by using the same symbol $m(s_{\alpha})$ to denote both the
left monoid action by a root $\alpha \in \Pi_{\fk}$ on 
$B_{n-1}\backslash \B$ and for the left monoid action on $B_{n-1}\backslash K/P.$   This may require some parsing
on the part of the reader.   For example, if we consider a $B_{n-1}$-orbit
$\Orb(Q_{\fr},Q_{\fl})$ on $\B$, then we will prove
that in  certain circumstances, $m(s_{\alpha})*\Orb(Q_{\fr},Q_{\fl}) = \Orb(m(s_{\alpha})*Q_{\fr}, Q_{\fl}).$  In this equality, on the left hand side of the equality,
the monoid action is the left monoid action on $B_{n-1}\backslash \B$ and on the
right hand side of the equality, the monoid action is the left monoid action
on $B_{n-1}\backslash K/P.$
\end{rem}

\subsection{Computation of Monoid Actions } \label{ss:monoidcomputation}

It follows from definitions that if $\alpha\in\Pi_{\fk}$ and $Q\subset Q_{K}$, then $\ms*Q\subset Q_{K}$.  
We now examine the effect of the left monoid action on the fibre bundle structure of 
$Q=\Orb(Q_{\fr}, Q_{\fl})\subset Q_{K}$ after first recalling some notation
from Remark \ref{r:goodreps}.  
Suppose $Q_{K}=K\cdot \fb$, where $\fb=\fb_{\hv}=\Ad(\hv)\fb_{+}\in\B$ is the representative given in Propositions \ref{prop:typeAflag}, \ref{prop_sooddflag}, and \ref{prop_soevenflag} respectively.  Let $\fr\supset\fb$ be the special parabolic subalgebra of $\fg$ associated to the $K$-orbit $Q_{K}$, and let $\fr=\fl\oplus\fu$ be its Levi decomposition.  Then $Q=B_{n-1}\cdot \Ad(\dotw)\Ad(\ell)\fb$, where $w\in W_{K}^{K\cap R}$ and $\ell\in K\cap L$ (see Equation (\ref{eq:goodrep})).  Further 
$Q_{\fr}=B_{n-1}\cdot w(\fr)$ and $Q_{\fl}=(B_{n-1}\cap L)\cdot\Ad(\ell)(\fb_{\hv}\cap \fl)$.  


\begin{thm}\label{thm:Kintertwiners}

Let $\alpha$ be a root in $\Pi_{\fk}$.  Let $\pi_{\fr}: G/B_{+}\to G/R$ be the  projection.  
Then  
\begin{equation}\label{eq:leftinvar}
\pi_{\fr}(\ms*\Orb(Q_{\fr}, Q_{\fl}))=\ms*Q_{\fr},
\end{equation}
where the monoid action on $B_{n-1}\backslash K/K\cap R$ in the right hand side
is given in Remark \ref{r:bnonemonoid}.

\noindent Suppose also that $\alpha\in w(\Phi_{\fl})$.  Then $\alpha$ is a simple root 
of $\fk\cap w(\fl)$ with respect to the Borel subalgebra $\fb_{n-1}\cap w(\fl)$ of $\fk\cap w(\fl)$, and 
\begin{equation}\label{eq:leftfibreaction}
\ms*\Orb(Q_{\fr}, Q_{\fl})=\mathcal{O}(Q_{\fr},m(s_{w^{-1}(\alpha)})*Q_{\fl}).  
\end{equation}
 Suppose $\alpha\not\in w(\Phi_{\fl})$. Then
\begin{equation}\label{eq:leftbaseaction}
\ms*\Orb(Q_{\fr}, Q_{\fl})=\mathcal{O}(\ms*Q_{\fr}, Q_{\fl}).
\end{equation}
\end{thm}
\begin{proof}

By Equation (\ref{eq:leftmonoidactionreal}), $\ms*Q$ is the open 
$B_{n-1}$-orbit in the variety $\mathcal{Y}:=P_{\alpha}\dotw\ell B/B$.  
By equivariance of $\pi_{\fr}$, its restriction $\pi_{\fr}:\mathcal{Y} \to
 P_{\alpha}\dotw R/R$ is open since quotient morphisms are open
(Section 12.2 of \cite{Hum}).
  Thus, $\ms*Q$ projects to the open $B_{n-1}$-orbit in $P_{\alpha}\dotw R/R$.   
Under the identification $Q_{K,\fr}\cong K/K\cap R$, the latter orbit is identified with 
the open $B_{n-1}$-orbit in $P_{\alpha} \dotw (K\cap R)/(K\cap R)$.  But this orbit is 
exactly $\ms *Q_{\fr}$ by Equation (\ref{e:smallleftmonoid}). 
Equation (\ref{eq:leftinvar}) follows.

Note that the $P_{\alpha}$-orbit $\mathcal{Y}$ fibres over 
$P_{\alpha}/(P_{\alpha} \cap w(R))$ via the identification
\begin{equation}\label{eq:Ybundle}
\mathcal{Y}\cong P_{\alpha}\times_{P_{\alpha}\cap w(R)} (P_{\alpha}\cap w(L)) \Ad(\dotw\ell)\fb.
\end{equation}
Indeed, the identification $\mathcal{Y} \cong P_{\alpha}\times_{P_{\alpha}\cap w(R)} (P_{\alpha}\cap w(R)) \Ad(\dotw\ell)\fb$ is formal, and since the diagonal torus of $K$
normalizes $P_{\alpha}, w(L),$ and $w(U)$, it follows that the 
action of $P_{\alpha} \cap w(R)$ on $\Ad(\dotw\ell)\fb$ is through
$P_{\alpha}\cap w(L)$.  
Now suppose that $\alpha$ is a root of $w(\fl)\cap\fk$.  
It follows that $P_{\alpha} w R/R=B_{n-1} w R/R$, whence $\ms*Q_{\fr}=Q_{\fr}$.  Further the fibre bundle in 
Equation (\ref{eq:Ybundle}) may be written as 
\begin{equation}\label{eq:Yfibreonly}
\mathcal{Y}\cong B_{n-1}\times_{B_{n-1}\cap w(R)} (P_{\alpha}\cap w(L)) \Ad(\dotw\ell)(\fb\cap\fl).    
\end{equation}
We observed in Equation (\ref{eq:isBorel}) that $B_{n-1}\cap w(L)$ is a Borel subgroup of 
the Levi subgroup $K\cap w(L)$ of $K$.  Since $P_{\alpha}\cap w(L)\supset B_{n-1}\cap w(L)$ and
$\alpha\in \Pi_{\fk}$, it immediately follows that $P_{\alpha}\cap w(L)$ is a parabolic 
subgroup of $K\cap w(L)$ and $\alpha$ is a simple root with respect to the Borel 
subalgebra $\fb_{n-1}\cap w(\fl)$ of $\fk \cap w(\fl)$. 

Equation (\ref{eq:Yfibreonly}) implies that 
\begin{equation}\label{eq:orbitfibre}
\ms*Q\cong B_{n-1}\times_{B_{n-1}\cap w(R)} \mathcal{O}, 
\end{equation}
where $\mathcal{O}$ is the open $B_{n-1}\cap w(R)$-orbit in the variety $(P_{\alpha}\cap w(L)) \Ad(\dotw\ell) (\fb\cap\fl)\subset \B_{w(\fl)}. $
Now $B_{n-1}\cap w(R) $ acts on $\B_{w(\fl)}$ through its image in the quotient $w(R)/w(U)\cong \B_{w(\fl)}$. 
In the proof of Theorem \ref{thm:Bkbundles}, we observed that this image is the Borel subgroup $B_{n-1} \cap w(L)$ of $w(L)\cap K$ (see 
Equation (\ref{eq:weclaim})).  Equation (\ref{eq:leftfibreaction}) now follows by translating by $\dotw^{-1}$, as is prescribed in the last comments of the proof of Theorem \ref{thm:Bkbundles},
and from Equations (\ref{eq:isBorel}) and (\ref{eq:leftmonoidactionreal}).


We now suppose that $\alpha$ is not a root of $\fk\cap w(\fl)$.  Then  $P_{\alpha}\cap w(L)=B_{n-1}\cap w(L)$.  
Thus, we can rewrite Equation (\ref{eq:Ybundle}) as 
$$
\mathcal{Y}\cong P_{\alpha}\times_{P_{\alpha}\cap\Ad(\dotw)R} (B_{n-1}\cap w(L) ) \Ad(\dotw\ell)\fb=P_{\alpha}\times_{P_{\alpha}\cap w(R)} \Ad(\dot w) Q_{\fl},
$$
where the last equality uses the definition of $Q_{\fl}.$   Therefore, the $B_{n-1}$-orbits in $\mathcal{Y}$ are in one-to-one correspondence with the 
$B_{n-1}$-orbits on the base $P_{\alpha} w R/R\cong P_{\alpha}  w (K\cap R)/(K\cap R)$. 
Thus, 
$$
\ms*Q=\pi_{\fr}|_{\mathcal{Y}}^{-1} (\ms* Q_{\fr})=\mathcal{O}(\ms *Q_{\fr}, Q_{\fl}),
$$
yielding (\ref{eq:leftbaseaction}).

\end{proof}

To understand how the bundle structure of $Q\subset Q_{K}$ behaves with respect to the right monoid action by roots $\alpha\in\Pi_{\fg}$, we need some preparation.


\begin{prop}\label{prop:Kmoves}
Let $\alpha\in\Pi_{\fg}$.  Suppose that $Q\subset Q_{K}$.  Then 
\begin{equation}\label{eq:monoidcontainment}
\ms*Q\subset \ms* Q_{K}
\end{equation}  
In particular, if $\ms*Q_{K}\neq Q_{K}$, then $\ms*Q\neq Q$.
\end{prop}
\begin{proof}
Let $Q=B_{n-1} x B_{n}/B_{n}$.  It follows from definitions 
that $\ms *Q_{K}$ is the unique open $K$-orbit in the variety $K x P_{\alpha}/B_{n}$.  
Since $B_{n-1} x P_{\alpha}/ B_{n}\subset K x P_{\alpha}/ B_{n}$, the subvariety
$\ms*Q_{K}\cap B_{n-1} x P_{\alpha}/ B_{n}$ is non-empty and open in $B_{n-1} x P_{\alpha} /B_{n}$.  
Since  $B_{n-1} x P_{\alpha}/ B_{n}$ is irreducible, it follows from Equation 
(\ref{eq:rightmonoid}) that $\ms*Q\cap \ms*Q_{K}\neq \emptyset.$  This implies equation 
(\ref{eq:monoidcontainment}).
\end{proof}


By Proposition \ref{prop:Kmoves},  if $\alpha\in\Pi_{\fg}$ and $\ms *Q_{K}=Q_{K}$, then $\ms* Q\subset Q_{K}$.  For such $\alpha$, we study  the monoid action by $\ms$ in terms of the fibre bundle description of the orbit $\Orb(Q_{\fr}, Q_{\fl})$ of $Q$.  But first we need to 
determine the roots $\alpha\in\Pi_{\fg}$ for which $\ms*Q_{K}=Q_{K}$.  
The following results are well-known and partly appear in tables on page 92 of
\cite{Collingwood}.
For the case of $\fg=\fgl(n)$, proofs are given in Examples 4.16 and 4.30 of \cite{CEexp} and for the type $B$ and $D$ cases, 
see Propositions 2.23 and 2.24 of \cite{CEspherical}.  
\begin{prop}\label{prop:roottypesforK}
\begin{enumerate}
\item Let $\fg=\fgl(n)$.  First, consider the case when $Q_{K}=Q_{i}$ is a closed 
$K$-orbit as in Part (2) of Proposition \ref{prop:typeAflag}.  Then the roots 
$\alpha_{i-1}$ and $ \alpha_{i}$ are non-compact for $Q_{K}$ and all other simple roots are compact.  Next, consider the case when $Q_{K}=Q_{i,j}$ as in Part (3) of Proposition \ref{prop:typeAflag}.
Then the roots $\alpha_{i-1}$ and $\alpha_{j}$ are complex stable for $Q_{K}$.  All other simple roots
$\alpha_{k}$ with $k\neq i, j-1$ are compact for $Q_{K}$.  If $j=i+1$, then the root $\alpha_{i}$ is
real for $Q_{K}$.  If $j>i+1$, the roots $\alpha_{i}$ and $\alpha_{j-1}$ are complex unstable.  
\item Let $\fg=\fso(2l+1)$.  If $Q_{K}$ is one of the two closed $K$-orbits given in Part (2) of 
Proposition \ref{prop_sooddflag}, then the root $\alpha_{l}$ is non-compact for $Q_{K}$, 
and all other simple roots are compact.  Consider the orbit $Q_{K}=Q_{i}$ as in Part (3) of Proposition \ref{prop_sooddflag}. Then 
the root $\alpha_{i}$ is complex stable for $Q_{K}$.  The root $\alpha_{i+1}$ is complex unstable
for $Q_{K}$ unless $i=l-1$ in which case it is real.  All other simple roots are compact for $Q_{K}$.  
\item Let $\fg=\fso(2l)$.  In the case $Q_{K}=Q_{+}$ is the unique closed orbit, then the roots
$\alpha_{l-1}$ and $\alpha_{l}$ are complex stable for $Q_{+}$ and all other simple roots are compact.  In the case $Q_{K}=Q_{i}$ as in Part (3) of Proposition \ref{prop_soevenflag}, then
the roots $\alpha_{i-1}$ and $\alpha_{i}$ are complex for $Q_{i}$ with $\alpha_{i-1}$ stable and $\alpha_{i}$ 
unstable.  All other simple roots are compact for $Q_{i}$, unless $i=l-1$ in which case $\alpha_{l}$ is also complex unstable. 
\end{enumerate}
\end{prop}

We use Proposition \ref{prop:roottypesforK} to study the relation between roots $\alpha\in\Pi_{\fg}$ 
such that $\ms*Q_{K}=Q_{K}$ and the special parabolic subalgebras $\fr\subset\fg$ constructed in 
Theorem \ref{thm:specials}.   Let $\tilde{w}\in W$ be given as follows:  
 \begin{equation}\label{eq:guyinW}
\tilde{w}=\left\lbrace\begin{array}{ccc} w_{i}=(n, n-1, \dots , i+1, i) \mbox{ if } \fg=\fgl(n), &  \mbox { and } &Q_{K}=Q_{i,j} \mbox{ or } Q_{i}\\
\mbox{ id }  \mbox{ if } \fg=\fso(n), & \mbox{ and }& Q_{K}=Q_{i} \mbox{ or } \,Q_{+}\\ 
 s_{\alpha_{l}}  \mbox{ if } \fg=\fso(2l+1), &\mbox{ and }  &Q_{K}=Q_{-}\end{array}\right\rbrace,
 \end{equation}
 and let $\dtw$ be a fixed representative of $\tilde{w}$.
Then it follows from the construction of the special parabolic subalgebra $\fr$ in Theorem 
\ref{thm:specials} that $\fr={\tilde{w}}(\fp_{S})$, where 
$\fp_{S}$ is a standard parabolic subalgebra of $\fg$ defined by the subset
$S \subset \Pi_{\fg}$  given by: 
\begin{equation}\label{eq:rootsforLevi}
S=\left\lbrace\begin{array}{cc}  \{\alpha_{i},\dots, \alpha_{j-1}\} \mbox{ for } \fg=\fgl(n), & Q_{K}=Q_{i,j}\\
\{\alpha_{i},\dots , \alpha_{l}\} \mbox{ for } \fg=\fso(2l), & Q_{K}=Q_{i}\\ 
\{\alpha_{i+1},\dots, \alpha_{l}\} \mbox{ for } \fg=\fso(2l+1), & Q_{K}=Q_{i}\end{array}\right\rbrace,
\end{equation}
and where $S=\emptyset$ and $\fr$ is a Borel if $Q_{K}$ is closed.
Proposition \ref{prop:roottypesforK} implies the following observation.  
\begin{cor}\label{c:Leviandroots}
Let $S\subset\Pi_{\fg}$ be the subset of simple roots given in
(\ref{eq:rootsforLevi}).  
Then 
\begin{enumerate}
\item If $\alpha\in S$, $\ms*Q_{K}=Q_{K}$.  
\item If $\alpha\in \Pi_{\fg}\setminus S$ and $\ms*Q_{K}=Q_{K}$, then 
$\alpha$ is compact for $Q_{K}$.  
\end{enumerate}
\end{cor}

We can now describe the right monoid action on a $B_{n-1}$-orbit $Q\subset Q_{K}$ by 
a simple root $\alpha\in \Pi_{\fg}$ with $\ms*Q_{K}=Q_{K}$.  Recalling notation from  Theorem \ref{thm:Kintertwiners}, we let $Q=\mathcal{O}(Q_{\fr}, Q_{\fl})=B_{n-1}\cdot \Ad(\dotw)\Ad(\ell \hat{v})\fb_+$, where $w\in W_{K}^{K\cap R}$ and $\ell\in K\cap L$, with
$Q_{\fr}=B_{n-1}\cdot w(\fr)$ and $Q_{\fl}=(B_{n-1}\cap L)\cdot\Ad(\ell)(\fb_{\hv}\cap \fl)$, and $\hat{v}$ is given in Propositions \ref{prop:typeAflag}, \ref{prop_sooddflag}, \ref{prop_soevenflag}

\begin{thm}\label{thm:rightbundleaction}

\noindent 

Suppose $\alpha\in\Pi_{\fg}$.  
\begin{enumerate}
\item If $\alpha\in S$, then 
\begin{equation}\label{eq:rightfibreaction}
\ms*\mathcal{O}(Q_{\fr}, Q_{\fl})=\mathcal{O}(Q_{\fr}, m(s_{\tilde{w}(\alpha)})*Q_{\fl}).  
\end{equation}
\item If $\alpha\in\Pi_{\fg}\setminus S$ and 
$\ms*Q_{K}=Q_K$, then  $\tilde{w}(\alpha)\in\Pi_{\fk}$ is a standard simple root strongly orthogonal to the roots of $S$, and 
\begin{equation}\label{eq:rightbaseaction}
\ms*\mathcal{O}(Q_{\fr}, Q_{\fl})=\mathcal{O}(m(s_{\tilde{w}(\alpha)})*Q_{\fr}, Q_{\fl}),
\end{equation}
where $m(s_{\tilde{w}(\alpha)})*Q_{\fr}$ is the right monoid action of the simple root $\tilde{w}(\alpha)$ on 
the $B_{n-1}$-orbit $Q_{\fr}$ in the partial flag variety $Q_{K,\fr}=K/K\cap R$ of $\fk$ defined in Part (2) of Remark \ref{r:bnonemonoid}.

\end{enumerate}
\end{thm}

\begin{proof}

First, suppose that $\alpha\in S$.  
Then by Corollary \ref{c:Leviandroots},  $\ms*Q_{K}=Q_{K}$.  Recall that 
by Equation (\ref{eq:rightmonoid}) $\ms*Q$ is the open $B_{n-1}$-orbit in the variety 
$\mathcal{Z}:=B_{n-1}\dotw\ell\hv P_{\alpha}/B_{+}$.   
Since $\alpha\in S$, $P_{\alpha}\subset P_{S}$.  Further, the representative 
$\hv=\dtw p$, where $\dtw$ is the above representative of $\tilde{w}$ given by (\ref{eq:guyinW}) and $p\in M_S$ (see Equations (\ref{eq:bij}), (\ref{eq:sooddboreli}), and (\ref{eq:soevenboreli})).  
It follows that the special parabolic subalgebra $\fr=\Ad(\hv)\fp_{S}$.  Thus, $\pi_{\fr}:G/B \to G/R$ maps  $\mathcal{Z}$ to $B_{n-1}\cdot w(\fr)=Q_{\fr}$ and hence 
$\pi_{\fr}:\mathcal{Z} \to Q_{\fr}$ is a $B_{n-1}$-equivariant fibre bundle. We call the fibre $F$, and
using 
 Equation \eqref{eq:Bkbundle} and Remark \ref{r:goodreps}, we can identify $F = (B_{n-1}\cap w(L))\cdot \Ad(\dw \ell \hat{v})(P_{\alpha}\cdot \fb_+)$. Using Equation \eqref{eq:isBorel}, we see $F$ is isomorphic to 
\begin{equation}\label{eq:fibreZ}
(B_{n-1}\cap L)\cdot \Ad(\ell\hv) (P_{\alpha} \cdot \fb_+) \cong (B_{n-1}\cap L) \ell\hv(P_{\alpha}\cap M_S)/ (B_{+}\cap M_S).
\end{equation}  
The open $B_{n-1}$-orbit in $\mathcal{Z}$ fibres over $Q_{\fr}$ with fibre isomorphic to the open $B_{n-1}\cap L$-orbit in the variety $F$.
Using the decomposition $\hv=\dtw p$, we can rewrite $F$ as 
\begin{equation}\label{eq:fibreZ2}
(B_{n-1}\cap L) \ell \tilde{w}p\tilde{w}^{-1} (P_{\tilde{w}(\alpha)} \cap L)/(\Ad(\tilde{w}) B_{+}\cap L).  
\end{equation}
But the open $B_{n-1}\cap L$-orbit in (\ref{eq:fibreZ2}) is exactly $m(s_{\tilde{w}(\alpha)})*Q_{\fl}$, using Equation (\ref{eq:rightmonoid}).  Equation (\ref{eq:rightfibreaction}) follows.  

Let $\alpha\in\Pi_{\fg}\setminus S$ with $\ms*Q_{K}=Q_{K}$.  
By (\ref{eq:rightmonoid}), $\ms*Q_{K}$ is the open $B_{n-1}$-orbit in the variety 
\begin{equation}\label{eq:Yvariety}
\mathcal{Y}=B_{n-1}\cdot\Ad(\dotw\ell\hv) (P_{\alpha}\cdot \fb_+)=B_{n-1}\cdot\Ad(\dotw\ell) (P_{\Ad(\hv)\alpha}\cdot \fb).
\end{equation}
We begin by analyzing the homogeneous space $P_{\Ad(\hv)\alpha} /B$.   
First, consider the root $\Ad(\hv)\alpha$ of $\Ad(\hv)\fh_{n}$.  We claim 
\begin{equation}\label{eq:rootclaim}
\Ad(\hv)\alpha=\tilde{w}(\alpha)\in \Pi_{\fk},
\end{equation} 
where $\tilde{w}\in W$ is given by (\ref{eq:guyinW}).  
We show (\ref{eq:rootclaim}) in the case where $\fg=\fgl(n)$.    
If $Q_{K}=Q_{i}$ is a closed $K$-orbit as in Part (2) of Proposition \ref{prop:typeAflag}, 
then $\hv=w_{i}=\tilde{w}$.  Since $\ms*Q_{K}=Q_{K}$, then by Proposition \ref{prop:roottypesforK}, 
$\alpha=\alpha_{j},$ $j\neq i-1,\, i$.  It is routine to check that 
$\tilde{w}(\alpha_{j})=w_{i}(\alpha_{j})=\alpha_{j-1}$ for $i<j\leq n$ and $\tilde{w}(\alpha_{j})=\alpha_{j}$ for $j<i-1.$  
In either case, $\tilde{w}(\alpha)\in\Pi_{\fk}=\{ \alpha_1, \dots, \alpha_{n-2} \}$.   
Now suppose $Q_{K}=Q_{i,j}$ is not closed as in Part (3) of 
Proposition \ref{prop:typeAflag}.  Then $\hv=w_{i}u_{\alpha_{i+1}}s_{\alpha_{i+1}}\dots s_{\alpha_{j-1}}$ by Equation (\ref{eq:bij}).  
Then since $\alpha\in \Pi_{\fg}\setminus S$ with $\ms*Q_{K}=Q_{K}$, then it follows from 
Definition \ref{dfn:strongortho} and Proposition \ref{prop:roottypesforK} that $\alpha$ is strongly orthogonal to the roots of 
$S=\{\alpha_{i},\dots, \alpha_{j-1}\}$ (see Equation (\ref{eq:rootsforLevi})).  
In particular, $\alpha$ is orthogonal to the roots of $S$, so 
$\Ad(\hv)\alpha=\tilde{w}u_{\alpha_{i}}(\alpha)$.   Further, 
since $\alpha$ is strongly orthogonal to $\alpha_{i}$, Proposiiton 6.72 of \cite{Knapp02} implies
that $u_{\alpha_{i}}(\alpha)=\alpha$, and we conclude that $\Ad(\hv)\alpha=\tilde{w}(\alpha)$.  The remainder of the argument proceeds as 
in the case where $Q_{K}=Q_{i}$ is closed, and the claim in (\ref{eq:rootclaim}) is established in the $\fgl(n)$ case.   The orthogonal cases are simpler and can be handled in an analogous fashion.

Now it follows from Part (2) of Corollary \ref{c:Leviandroots} that $\alpha$ is compact for $Q_{K}$.  It then follows
from our discussion in Section \ref{ss:Korbits} that the space 
$P_{\tilde{w}(\alpha)}/B$ is a single orbit of a Levi factor of $P_{\tilde{w}(\alpha)}$. 
Let  $M_{\alpha}\subset P_{\alpha}$ be the $H_n$-stable Levi factor, so that
$M_{\tilde{w}(\alpha)}:=\Ad(\tilde{w})M_{\alpha}\subset K$ is a Levi factor of $P_{\tilde{w}(\alpha)}$.  Thus, 
the variety $\mathcal{Y}$ in Equation (\ref{eq:Yvariety}) equals $\mathcal{Y}=B_{n-1}\cdot \Ad(\dotw \ell) (M_{\tilde{w}(\alpha)}\cdot \fb$).  
Since $\alpha$ is strongly orthogonal to the roots of $S$, it follows that $\tilde{w}(\alpha)$ is strongly orthogonal 
to the roots of $\fl$. Since $L\cap K\subset L$, it follows that the groups $M_{\tilde{w}(\alpha)}$ and $L\cap K$ commute.   
Therefore, 
\begin{equation}\label{eq:Yimage}
\pi_{\fr}(\mathcal{Y})=B_{n-1} w M_{\tilde{w}(\alpha)} (K\cap R)/ (K\cap R).
\end{equation}
If $J\subset \Pi_{\fk}$, we let $S_{J}\subset K$ be the standard parabolic subgroup of $K$ given by the subset 
$J$ of $\Pi_{\fk}$.  Now we claim that 
\begin{equation}\label{eq:varietiesagree}
\pi_{\fr}(\mathcal{Y})=B_{n-1} w S_{\Pi_{\fl\cap\fk}\cup \tilde{w}(\alpha)} (K\cap R)/(K\cap R),
\end{equation}
Indeed,
\begin{equation*}
\begin{split}
B_{n-1} w S_{\Pi_{\fl\cap\fk}\cup \tilde{w}(\alpha)}(K\cap R)/(K\cap R)&=B_{n-1} w S_{\tilde{w}(\alpha)}(K\cap R)/(K\cap R)\\
 &=B_{n-1} w M_{\tilde{w}(\alpha)}(K\cap R)/(K\cap R)\\
 &=\pi_{\fr}(\mathcal{Y}).
\end{split}
\end{equation*}

Now by Equation (\ref{eq:smallrightmonoid}), $m(s_{\tilde{w}(\alpha)})*Q_{\fr}$ is the open $B_{n-1}$-orbit 
in the variety $B_{n-1} w  S_{\Pi_{\fl\cap\fk}\cup \tilde{w}(\alpha)} (K\cap R)/(K\cap R)=\pi_{\fr}(\mathcal{Y})$.   
Since the group $M_{\tilde{w}(\alpha)}$ acts trivially on $\fl\cap\fk$,  the variety 
$\mathcal{Y}$ has the structure of a $B_{n-1}$-equivariant fibre bundle with smooth base $\pi_{\fr}(\mathcal{Y})$ and 
fibre $Q_{\fl}$.  It follows that $\pi_{\fr}|_{\mathcal{Y}}$ is a flat morphism and is therefore open by Exercise III.9.1 of \cite{Ha}. 
Thus, $\pi_{\fr}(\ms*Q)=m(s_{\tilde{w}(\alpha)}) * Q_{\fr}$ and Equation (\ref{eq:rightbaseaction}) now follows. 
\end{proof}

It follows from Theorems \ref{thm:Kintertwiners} and \ref{thm:rightbundleaction} that 
to understand the monoid action on $Q=\Orb(Q_{\fr}, Q_{\fl})\subset Q_{K}$ by roots $\alpha\in\Pi_{\fg}\cup\Pi_{\fk}$ such that 
$\ms*Q_{K}=Q_{K}$, it suffices to understand the monoid action 
on $B_{n-1}$-orbits in $K/(K\cap R)$ discussed in Remark \ref{r:bnonemonoid}
as well as the monoid action on $Q_{\fl}$.  The former is well understood, and in the next result, we use 
 the correspondence $Q \mapsto Q^{op}$ in Equation (\ref{eq:correspondence2}) to understand the latter.  
For the statement of the next result, we recall the particular representative $\tilde{\fb}$ of the open $K$-orbit that
is used in the proof of Theorem \ref{thm:openorbit} to establish the correspondence in (\ref{eq:correspondence2}).  
We let $\tildeQK=K\cdot \tilde{\fb}$ where $\tilde{\fb}$ is the Borel subalgebra of 
$\fg$ which stabilizes the flag $\tilde{\mathcal{F}}$: 
\begin{equation} \label{eq:secondopenflag}
\tilde{\mathcal{F}}=\left\lbrace\begin{array}{cc} (e_{n-1}+e_{n}\subset e_{1}\subset e_{2}\subset\dots\subset e_{n-2}\subset e_{n}) & \mbox{ if } \fg=\fgl(n),\\
(e_{l}\subset e_{1}\subset e_{2}\subset\dots\subset e_{l-1}) & \mbox{ if } \fg=\fso(2l)     \\ 
 (\imath e_{l}+\sqrt{2}\imath e_{0}-\imath e_{-l}\subset e_{1}\subset e_{2}\subset \dots \subset e_{l-1})  & \mbox{ if } \fg=\fso(2l+1) \end{array}\right\rbrace,
\end{equation}


\begin{thm}\label{thm:openinter}
Let $Q=B_{n-1}\cdot \Ad(k)\tilde{\fb}$ be a 
$B_{n-1}$-orbit in the open $K$-orbit ${\tildeQ}_K$ on $\B_{n}$, 
and let $Q^{op}=B_{n-2}\cdot \Ad(k^{-1})\fb_{n-1}$ be the corresponding 
$B_{n-2}$-orbit on $\B_{n-1}$ as in Remark \ref{r:theycorrespond}.  

Let $\alpha\in\Pi_{\fk}$, then  
\begin{equation}\label{eq:firstequiv}
(\ms*Q)^{op} = \ms*Q^{op}\subset\B_{n-1},
\end{equation}
where $\ms*Q$ denotes the left monoid action of $s_{\alpha}$ on the $B_{n-1}$-orbit $Q$ in $\B_n$, 
and $\ms*Q^{op}$ denotes the right monoid action of $s_{\alpha}$ on the $B_{n-2}$-orbit $Q^{op}$ in $\B_{n-1}$.

Let $\alpha=\alpha_{j}\in \Pi_{\fg}$ be a simple root which is compact for ${\tilde{Q}}_K$.  
Then  
\begin{equation}\label{eq:secondequiv}
(m(s_{\alpha_{j}})*Q)^{op} = m(s_{\alpha_{j-1}})*Q^{op}, 
\end{equation}
where $m(s_{\alpha_{j}})*Q$ is the right monoid action of $s_{\alpha_{j}}$ on $Q$ 
and $m(s_{\alpha_{j-1}})*Q^{op}$ is the left monoid action of $s_{\alpha_{j-1}}$ on $Q^{op}$ and $\alpha_{j-1}\in\Pi_{\fk_{n-2}}$ is a standard simple root.


Moreover, the correspondences in Equations (\ref{eq:firstequiv}) and (\ref{eq:secondequiv})
preserve root types.  That is to say that $\alpha$ is complex stable, complex unstable, non-compact, etc
for $Q$ if and only if $\alpha$ is complex stable, unstable, non-compact, etc for $Q^{op}$ in (\ref{eq:firstequiv}) 
and similarly for $\alpha_{j}$ and $\alpha_{j-1}$ in (\ref{eq:secondequiv}).


\end{thm}
 \begin{proof}
   Let $\alpha\in\Pi_{\fk}$.  Then by Equation \ref{eq:leftmonoidactionreal},
     we know $\ms*Q$ is the open $B_{n-1}$-orbit 
 in the variety $\mathcal{Y}_{Q}:=P_{\alpha}kB_{n-2}/B_{n-2}\subset \tilde{Q}$.  
 Under the correspondence in Equation (\ref{eq:correspondingsubvariety}), we have 
 $$
 \mathcal{Y}_{Q}^{op}=B_{n-2}k^{-1}P_{\alpha} /B_{n-1}.
 $$
 Since $\ms*Q^{op}$ is the open $B_{n-2}$-orbit in $B_{n-2}k^{-1}P_{\alpha} /B_{n-1}$, Equation (\ref{eq:firstequiv}) now 
 follows from Equation  (\ref{eq:rightmonoid}).

   
 Now we consider the right monoid action.  Let $\alpha\in\Pi_{\fg}$ be compact for the open $K$-orbit $\tildeQK=K\cdot \tilde{\fb}$.  Let $\tilde{\fb}=\Ad(g)\fb_{+}$.  
 By Equation (\ref{eq:rightmonoid}), $\ms*Q$ is the open $B_{n-1}$-orbit in the variety 
 $B_{n-1}\cdot \Ad(kg)(P_{\alpha}\cdot \fb_+)$. Since $\alpha$ is compact for $\tildeQK$,  $g P_{\alpha}/B_{+}$ is a single $K\cap \Ad(g)P_{\alpha}$-orbit, so that 
 $gP_{\alpha}/B_{+}\cong (K\cap \Ad(g)P_{\alpha})/ (K\cap \Ad(g) B_{+}).$  Therefore, 
 \begin{equation}\label{eq:redbox} 
 B_{n-1}\cdot \Ad(kg)(P_{\alpha}\cdot \fb_+) \cong B_{n-1}\cdot \Ad(k) ((\Ad(g)P_{\alpha}\cap K)\cdot\fb_{n-2}).  
 \end{equation}
 Using (\ref{eq:correspondingsubvariety}) and (\ref{eq:redbox}), it follows that 
 \begin{equation}\label{eq:redbox2}
 (B_{n-1}\cdot \Ad(kg)(P_{\alpha}\cdot \fb_{+}))^{op}= (\Ad(g)P_{\alpha}\cap K )\cdot \Ad(k^{-1}) \fb_{n-1}.  
 \end{equation} 
We now claim that the group $\Ad(g)P_{\alpha}\cap K$ in (\ref{eq:redbox2}) is the simple parabolic subgroup 
of $K_{n-2}$ corresponding to the simple root $\alpha_{j-1}$.  Equation (\ref{eq:secondequiv}) now follows from (\ref{eq:redbox2}).  We verify the claim in the case where $\fg=\fgl(n)$.  Suppose $\alpha=\alpha_{j}$ is compact for $\tildeQK =Q_{1,n}$.  By Proposition \ref{prop:roottypesforK}, it follows that  $j\neq 1, n-1$.  The image of the flag from Equation  (\ref{eq:secondopenflag}) is the partial flag $\mathcal{P}\mathcal{F}_z$ denoted $e_{n-1} + e_n \subset e_1 \subset \dots \subset e_{j-2} \subset\{e_{j-1}, e_j\} \subset e_{j+1} \subset \dots \subset e_{n-2}\subset e_n $. As in the proof of Theorem \ref{thm:openorbit}, we compute the stabilizer  $G_{n-1,z}$ of this partial flag in $G_{n-1}$.  For $g$ in $G_{n-1,z}$, since $g$ stabilizes the line spanned by $e_{n-1}+e_n$, we see that $g \cdot e_{n-1}=e_{n-1}$.   Recalling the notation setting $U_k$ equal to the span of $e_1, \dots, e_k$, we see as before that $g\cdot U_k \subset U_{k}+\C(e_{n-1} + e_n)$ for $k=1, \dots, n-2$ and $k \not= j-1$, and for $k=j-1$, $g\cdot U_{j-1} \subset U_{j}+\C(e_{n-1} + e_n).$  As in Theorem \ref{thm:openorbit}, since $g \in G_{n-1}$ and hence stabilizes the subspace $U_{n-1}$, we see that $g \in G_{n-2}$  and the claim follows since the subgroup of $G_{n-2}$ stablilizing the partial flag
$(e_1 \subset \dots \subset \{ e_{j-1}, e_{j} \} \subset \dots \subset e_{n-2})$ is the simple parabolic subgroup of $G_{n-2}$ corresponding to $\alpha_{j-1}$.  The proofs of the claim in the orthogonal cases are similar, and are left to the reader.

To see the assertion about root types, consider the variety $\mathcal{Y}_{Q}=p^{-1}(K\cdot (Q,eB_{n-1}))$, where $p:G/B \times K/B_{n-1} \to G/P_{\alpha} \times K/B_{n-1}$ in case $\alpha \in \Pi_{\fg}$ and $p:G/B \times K/B_{n-1} \to G/B \times K/P_{n-1, \alpha}$ in case $\alpha \in \Pi_{\fk}$.  For $\alpha \in \Pi_{\fk}$,
   then $\alpha$ is complex stable for $Q$ if and only if $\mathcal{Y}_Q$ consists of two double cosets with $K\cdot (Q,eB_{n-1})$ connected and codimension 1 in the closure of the other double coset.  This is the same as the condition for $\alpha$ to be complex stable for $Q^{op}$.  The remaining cases are similar.

  \end{proof}

 \begin{rem}\label{r.monoidsummary}
   As a consequence of the above results, we have completely understood the monoid action for simple roots of $\fk$, and for simple roots of $\fg$ in the subset $S$ associated to the orbit $Q_K$ by \eqref{eq:rootsforLevi}, and for simple roots $\alpha$ of $\fg$ not in $S$ when $m(s_{\alpha})*Q_K=Q_K$.   We will treat the remaining cases in a sequel to this paper.
   \end{rem}

\bibliographystyle{amsalpha.bst}

\bibliography{bibliography-1}

\end{document}